\documentclass[a4paper]{article}

\usepackage{geometry}
\usepackage{enumerate}
\usepackage{indentfirst} 
\usepackage[utf8]{inputenc}
\usepackage{color}
\usepackage{mathtools}
\usepackage{amsthm,amssymb,amsfonts}
\usepackage{bm}
\usepackage{stmaryrd}
\usepackage[T1]{fontenc}
\usepackage{setspace}
\usepackage{graphicx}
\usepackage{tcolorbox}
\usepackage{listings}
\usepackage{float}
\usepackage{subcaption}
\usepackage{geometry}
\usepackage{wrapfig}
\usepackage[bookmarks=true,bookmarksnumbered=true,bookmarksopen=true,bookmarksopenlevel=2,backref=false]
 {hyperref}
\hypersetup{
colorlinks,
linkcolor=black,
anchorcolor=blue,
citecolor=olive
}
\usepackage{makeidx}
\usepackage{aliascnt}
\usepackage[nameinlink]{cleveref}
\usepackage{mathrsfs}
\usepackage{booktabs}
\usepackage{autobreak}
\allowdisplaybreaks

\usepackage[shortlabels, inline]{enumitem}
\setenumerate[1]{itemsep=0pt,partopsep=0pt,parsep=\parskip,topsep=5pt}
\setitemize[1]{itemsep=0pt,partopsep=0pt,parsep=\parskip,topsep=5pt}
\setdescription[1]{itemsep=0pt,partopsep=0pt,parsep=\parskip,topsep=5pt}

\let\emph\relax  
\DeclareTextFontCommand{\emph}{\bfseries\em}

\makeatletter
\newcommand\RedeclareMathOperator{%
  \@ifstar{\def\rmo@s{m}\rmo@redeclare}{\def\rmo@s{o}\rmo@redeclare}%
}
\newcommand\rmo@redeclare[2]{%
  \begingroup \escapechar\m@ne\xdef\@gtempa{{\string#1}}\endgroup
  \expandafter\@ifundefined\@gtempa
     {\@latex@error{\noexpand#1undefined}\@ehc}%
     \relax
  \expandafter\rmo@declmathop\rmo@s{#1}{#2}}
\newcommand\rmo@declmathop[3]{%
  \DeclareRobustCommand{#2}{\qopname\newmcodes@#1{#3}}%
}
\@onlypreamble\RedeclareMathOperator
\makeatother

\RedeclareMathOperator{\Re}{Re}
\RedeclareMathOperator{\Im}{Im}

\theoremstyle{plain}
\newtheorem{theorem}{Theorem}[section]
\newtheorem*{theorem*}{Theorem}

\newaliascnt{lemma}{theorem}

\aliascntresetthe{lemma}

\newtheorem*{lemma*}{Lemma}

\newtheorem*{corollary*}{Corollary}
\newtheorem{proposition}[theorem]{Proposition}
\newtheorem*{proposition*}{Proposition}

\newtheorem*{axiom*}{Axiom}

\newtheorem*{property*}{Property}

\theoremstyle{definition}

\newtheorem*{definition*}{Definition}
\newtheorem{example}{Example}[section]
\newtheorem*{example*}{Example}

\newtheorem*{problem*}{Problem}

\newtheorem*{discussion*}{Discussion}

\newtheorem*{notation*}{Notation}

\newtheorem*{assumption*}{Assumption}

\newtheorem*{task*}{Task}

\newtheorem*{idea*}{Idea}
\newtheorem{conjecture}{Conjecture}[section]
\newtheorem*{conjecture*}{Conjecture}

\theoremstyle{remark}
\newtheorem{remark}{Remark}[section]
\newtheorem*{remark*}{Remark}

\newtheorem*{claim*}{Claim}

\newtheorem*{fact*}{Fact}

\usepackage{algorithm}
\usepackage{algpseudocode}
\usepackage{amsmath}

\usepackage{listings}
\usepackage{xcolor}

\lstdefinestyle{lfonts}{
	upquote=true,
	basicstyle   = \linespread{1.0}\footnotesize\ttfamily,
	stringstyle  = \color{brown!50!black},
	identifierstyle = \color{black}, 
	keywordstyle = \color{blue!50!black}\bfseries,
	commentstyle = \color{olive}\rmfamily\itshape,
}
\lstdefinestyle{lnumbers}{
	numbers     = left,
	numberstyle = \scriptsize\sffamily\color{black!50},
	numbersep   = 1em,
	firstnumber = 1,
	stepnumber  = 1,
}
\lstdefinestyle{llayout}{
	breaklines       = true,
	tabsize          = 4,
	columns          = flexible,
	lineskip		 = 0pt,
}
\lstdefinestyle{lgeometry}{
	xleftmargin      = 20pt,
	xrightmargin     = 0pt,
	frame            = tb,
	framesep         = \fboxsep,
	framexleftmargin = 20pt,
}
\lstdefinestyle{lgeneral}{
	style = lfonts,
	style = lnumbers,
	style = llayout,
	style = lgeometry,
}

\usepackage{setspace}
\usepackage{extarrows}

\usepackage{adjustbox}
\usepackage{booktabs}
\usepackage{longtable}
\usepackage{nicematrix}
\NiceMatrixOptions{
  custom-line = { letter = : , command = vdashline , tikz = dashed }, 
  custom-line = { letter = ; , command = hdashline , tikz = dashed }  
}
\usepackage{standalone} 
\usepackage{subcaption}

\newcommand{\Span}{\operatorname{span}}

\usepackage{bm}
\newcommand*{\tensor}[1]{\mathbf{#1}}
\renewcommand*{\matrix}[1]{\mathbf{#1}}
\renewcommand*{\vector}[1]{\bm{#1}}

\usepackage{mathtools}
\usepackage{makecell}

\usetikzlibrary{patterns, decorations.pathreplacing}
\usepackage{contour}
\contourlength{1.2pt}

\usepackage{forest}

\makeatletter
\newdimen\myshiftaligndimen
\forestset{
    declare dimen register={shift align},
    shift align'=0pt,
    default preamble={
        TeX={%
          \pgfinterruptpicture
          \myshiftaligndimen=\fontdimen2\font@name
          \global\myshiftaligndimen\myshiftaligndimen
          \endpgfinterruptpicture
        },
        /tikz/baseline={([yshift=-0.75\myshiftaligndimen]current bounding box.center)},
        for tree={%
            grow=north,
            content=,
            circle,
            fill,
            minimum size=2.5pt,
            inner sep=0.5pt,
            l=3pt,
            l sep=3pt,
            s sep=2pt,
            edge={-, thick},
        },
    },
    midtree/.style={
         /tikz/every node/.append style={font=\small},
        for tree={%
            label/.option=content,
            grow=north,
            content=,
            circle,
            fill,
            minimum size=3pt,
            inner sep=0pt,
            l=8pt,
            l sep=8pt,
            s sep=9pt,
            edge={-, very thick},
        },
    },
    bigtree/.style={
        for tree={%
            label/.option=content,
            grow=north,
            content=,
            circle,
            fill,
            minimum size=4pt,
            inner sep=0.5pt,
            l=20pt,
            l sep=12pt,
            s sep=12pt,
            edge={-, very thick},
        },
    },
    EL/.style 2 args={
        edge label={node[midway, font=\sffamily\scriptsize, #1]{#2}},
    },
}
\makeatother

\usepackage{biblatex}

\title{Low Stage and High Order Explicit Runge--Kutta Methods via $Q$- and $D$-Conditions: Several Construction Details}
\author{Junyuan He and Jizu Huang}

\begin{document}
\maketitle

\begin{abstract}
This note provides additional details on the construction of the $Q$/$D$-space framework for sufficient order conditions of explicit Runge--Kutta (ERK) methods. Specifically, it presents a general version of the sufficiency theorem, several examples illustrating the verification of the sufficient conditions, a detailed construction of an ERK scheme of order $p=10$, the assembly of the associated linear systems, a complexity analysis of the construction algorithm, and tables of coefficients for the constructed ERK methods.
\end{abstract}

\section{Another sufficient condition}\label{Another sufficient condition}

We begin by recalling Theorem 3.1 in \cite{he2026low} as follows.

\begin{theorem}\label{thm:my_sufficiency_thm}
    An ERK method has order at least $p$ if the following conditions hold for some $m, n$ with $m \geq n-1$ and $m + n + 1 \geq p$:
    \begin{enumerate}
        \item Quadrature conditions for $[{\,\begin{forest} [] \end{forest}\,}^{k-1}]$, denoted $B(p)$:
        \begin{equation*}
            \vector{b} \cdot \vector{c}^{\odot (k-1)} = \frac{1}{k}, \quad k = 1,\, \ldots,\, p.
        \end{equation*}
        \item Q-orthogonality conditions, denoted $QO(m)$:
        \begin{equation*}
            \vector{b} \odot Q_m = \{ \vector{0} \}.
        \end{equation*}
        \item D-orthogonality conditions, denoted $DO(n)$: 
        \begin{equation*}
            D_n \cdot \vector{c}^{\odot (k-1)} = \{ 0 \}, \quad k = 1,\, \ldots, p-n.
        \end{equation*}
        \item Q-and-D mutual orthogonality conditions, denoted $QD(m,n)$:
        \begin{equation*}
            Q_m \odot D_n = \{ \vector{0} \}.
        \end{equation*}
        \item Q-ring condition (the space $Q_m$ form a ring under Hadamard product), denoted $QR(m)$: If $m_2 \leq m_1 \leq m$, 
        \begin{equation*}
            Q_{m_1} \odot Q_{m_2} = Q_{m_1}.
        \end{equation*}
    \end{enumerate}
\end{theorem}

Next, we present another sufficient condition for ERK methods, which can be viewed as a more general version of the sufficiency theorem compared to Theorem 3.1 in \cite{he2026low}.

\begin{theorem}\label{thm:my_sufficiency_thm_general}
    An ERK method has order at least $p$ if the following conditions hold for some $m, n$ with $m \geq n-1$ and $m + n + 1 \geq p$.
    \begin{enumerate}
        \item Quadrature conditions for $[{\,\begin{forest} [] \end{forest}\,}^{k-1}]$, denoted $B(p)$:
        \begin{equation*}
            \vector{b} \cdot \vector{c}^{\odot (k-1)} = \frac{1}{k}, \quad k = 1,\, \ldots,\, p.
        \end{equation*}
        \item Q-orthogonality conditions, denoted $QO(m)$:
        \begin{equation*}
            \vector{b} \odot Q_m = \{ \vector{0} \}.
        \end{equation*}
        \item D-orthogonality conditions, denoted $DO(n)$: 
        \begin{equation*}
            D_n \cdot \vector{c}^{\odot (k-1)} = \{ 0 \}, \quad k = 1,\, \ldots, p-n.
        \end{equation*}
        \item Weak Q-and-D mutual orthogonality conditions, denoted $QD_{\mathrm{weak}}(m,n)$:
        \begin{equation*}
            Q_m \cdot D_n = \{ 0 \}.
        \end{equation*}
        \item Pivot residual condition, denoted $PR(n)$: 
        Define the space of pivot residual vectors:
        \begin{equation*}
            W_n = \{ (\vector{q} \odot \vector{d}) \mathop{\times^1} \matrix{A} - \vector{q} \odot (\vector{d} \mathop{\times^1} \matrix{A}) \mid \vector{q} \in Q_{l_1},\; \vector{d} \in D_{l_2},\; l_1 + l_2 \leq n-1  \}.
        \end{equation*}
        The $PR(n)$ condition is that 
        \begin{equation}
            W_n \cdot \left\{ \vector{\Phi}(t) \mid m+1 \leq |t| \leq p-3 \right\} = \{ 0 \}.
        \end{equation}
        \item Q-ring condition (the space $Q_m$ form a ring under Hadamard product), denoted $QR(m)$: If $m_2 \leq m_1 \leq m$, 
        \begin{equation*}
            Q_{m_1} \odot Q_{m_2} = Q_{m_1}.
        \end{equation*}
    \end{enumerate}
\end{theorem}

\begin{proof}
    The proof is essentially the same as that of \Cref{thm:my_sufficiency_thm} except case 2.1 and case 2.2, in which the weak $Q$-and-$D$ mutual orthogonality conditions prevents direct reduction, and the pivot residual condition $PR(n)$ is used to proceed with the reduction and show that the remainder term vanishes. We give a sketch of the argument for these two cases.

    Since the $QD(m,n)$ condition is replaced by the $QD_{\mathrm{weak}}(m,n)$ condition, the following remainder term in Case 2.1 
    \begin{equation*}
        R_{12} = \vector{d}_k^\top (\vector{q} \odot \vector{c}^{\odot k'})
    \end{equation*}
    may no longer vanish for some $\vector{q} \in Q_m$.Since $m+n+1 \geq p$,  we can rewrite it as 
    \begin{equation*}
        R_{12} = \underbrace{(\vector{d}_k \odot \vector{c}^{\odot k'_1})}_{{\in D_n}} \cdot \underbrace{(\vector{q} \odot \vector{c}^{\odot k'_2})}_{{\in Q_m}}  = 0.
    \end{equation*}
    Therefore, $R_{12}$ also vanishes by the $QD_{\mathrm{weak}}(m,n)$ condition (illustrated in \Cref{fig:general-case2.1}).

    In case 2.2, for the remainder term $R_2$,
    we still perform the $\vector{q}$-reduction using Lemma 3.2 on the subtrees $t_{1,1}, \dots, t_{1,r_{1}}$, where $|t_{1,1}| > m$ and $|t_{1,j}| \leq m$ for $j=2, \dots, r_{1}$, resulting the following three types of terms: 
    \begin{align*}
        \text{Principal term:\quad} & \vector{d}_k^\top (\matrix{A}\vector{\Phi}(t_{1,1}) \odot \vector{c}^{\odot k''}), \\
        \text{Remainder term (type 1):\quad} & R_{21} = \vector{d}_k^\top (\matrix{A}\vector{\Phi}(t_{1,1}) \odot \vector{q}), \\
        \text{Remainder term (type 2):\quad} & R_{22} = \vector{d}_k^\top \big(\matrix{A}\vector{\Phi}(t_{1,1}) \odot (\tilde{\vector{q}} \odot \vector{c}^{\odot \tilde{k}''})\big),
    \end{align*}
    for some $\vector{q}, \tilde{\vector{q}} \in Q_m$. 
    The pricipal term of this reduction is dealt with the same way as \Cref{thm:my_sufficiency_thm}, but the remainder terms are not immediately zero because the $QD(m,n)$ condition is not applicable. To deal with the remainder terms, we first claim that the type 2 terms can be transformed to type 1 terms. Define $\tilde{\vector{q}}' := \tilde{\vector{q}} \odot \vector{c}^{\odot \tilde{k}''}$. Since $|t_{1,1}| \ge m+1$, by some simple calculation we get $\tilde{\vector{q}}' \in Q_m$. This becomes exactly the type 1 term. Now we only need to consider the type 1 terms, which can be written as the form
    \begin{align*}
      R_{21} &= \left( (\vector{q} \odot \vector{d}_k) \mathop{\times^1} \matrix{A} \right) \cdot \vector{\Phi}(t_{1,1}) \\
      & = \left( (\vector{q} \odot \vector{d}_k) \mathop{\times^1} \matrix{A} \right) \cdot \left( \bigodot_{j=1}^{r_{1,1}} \matrix{A}\vector{\Phi}(t_{1,1,j}) \right)
    \end{align*}
    for some $\vector{q} \in Q_{l_1}$ and $l_1 \leq p-k-2-|t_{1,1}| \leq p-k-m-3 \leq n-k-2$. By definition, $\vector{d}_k \in D_{l_2}$ with $l_2 = k+1$, so $l_1 + l_2 \leq n-1$. Hence the pivot residual condition $PR(n)$ applies:
    there exists some $\vector{w} \in W_n$ such that
    \begin{equation*}
        \vector{w} = (\vector{q} \odot \vector{d}_k) \times^1 \matrix{A} - \vector{q} \odot (\vector{d}_k \times^1 \matrix{A}).
    \end{equation*}
    So, $R_{21}$ splits into two terms:
    \begin{equation*}
        R_{21} = \vector{w} \cdot \left( \bigodot_{j=1}^{r_{1,1}} \matrix{A}\vector{\Phi}(t_{1,1,j}) \right) 
        + \left( \vector{q} \odot \underbrace{(\vector{d}_k \times^1 \matrix{A})}_{\in D_{k+1}} \right) \cdot \left( \bigodot_{j=1}^{r_{1,1}} \matrix{A}\vector{\Phi}(t_{1,1,j}) \right).
    \end{equation*}
    The first term vanishes by the $PR(n)$ condition. 
    Through the recursive use of the $PR(n)$ condition and the $\vector{q}$-reduction, the second term eventually becomes $\vector{d}_n^\top \vector{q} $ for some $\vector{q} \in Q_m$, which vanishes due to the $QD_{\mathrm{weak}}(m,n)$ condition. 

    By repeating the proof of \Cref{thm:my_sufficiency_thm}, the proof of this theorem can be completed.
\end{proof}

\begin{figure}[!ht]
    \centering
    \begin{forest} bigtree
      [, label=below:{\small $\vector{d}_k \mathop{\times^1} \tensor{I}_s^3$}
        [ , label=above:{\small $\vector{c}^{\odot k}$}]
        [, label=above:{\small $\vector{q}$}, edge={very thick} ]
      ]
    \end{forest}
    $\quad = \quad$
    \begin{forest} bigtree
      [, label=below:{\small $\vector{d}_k \odot \vector{c}^{k_1}$}
        [, label=above:{\small $\vector{q} \odot \vector{c}^{k_2}$}, edge={very thick} ]
      ]
    \end{forest}
    \caption{Case 2.1: reduction to a weak $QD$ orthogonality term.}
    \label{fig:general-case2.1}
\end{figure}

\begin{figure}[!ht]
    \centering
    \begin{forest} bigtree
      [, label=below:{\small $\vector{d}_k \mathop{\times^1} \tensor{I}_s^3$}
        [ , label={[rotate=90, anchor=west]above:{\small $\vector{q} \in Q_m$}}]
        [, label={[rotate=90, anchor=west]above:{\small $\matrix{A} \vector{\Phi}(t_{1,1})$}}, edge={very thick}]
      ]
    \end{forest}
    $\quad = \quad $
    \begin{forest}
      bigtree
      [, label=below:{\small $\vector{d}_k \mathop{\times^1} \tensor{I}_s^3$}, s sep=18pt
        [ , label={[anchor=west]above:{\small $\vector{q}$}}]
        [, label=left:{\small $\matrix{A} \tensor{I}_s^4$}, s sep=12pt
            [ , label={[rotate=90, anchor=west]above:{\small $\matrix{A}\vector{\Phi}(t_{1,1,3})$}}, l=30pt] 
            [ , label={[rotate=90, anchor=west]above:{\small $\matrix{A}\vector{\Phi}(t_{1,1,2})$}}, l=30pt] 
            [ , label={[rotate=90, anchor=west]above:{\small $\matrix{A}\vector{\Phi}(t_{1,1,1})$}}, l=30pt]
        ]
      ]
    \end{forest}
    $\quad \xrightarrow{\text{definition of }W_n} \quad$
    \begin{forest} bigtree
      [, label=below:{\small $\vector{w} \mathop{\times^1} \tensor{I}_s^4$}, s sep=12pt
        [ , label={[rotate=90, anchor=west]above:{\small $\matrix{A}\vector{\Phi}(t_{1,1,3})$}}] 
        [ , label={[rotate=90, anchor=west]above:{\small $\matrix{A}\vector{\Phi}(t_{1,1,2})$}}] 
        [ , label={[rotate=90, anchor=west]above:{\small $\matrix{A}\vector{\Phi}(t_{1,1,1})$}}]
      ]
    \end{forest}
    $\quad + \quad$
    \begin{forest}
      bigtree
      [, label=below:{\small $\vector{d}_k$}
        [, label=left:{\small $\matrix{A} \tensor{I}_s^5$}, s sep=12pt
            [ , label={above:{\small $\vector{q}$}}, l=30pt]
            [ , label={[rotate=90, anchor=west]above:{\small $\matrix{A}\vector{\Phi}(t_{1,1,3})$}}, l=30pt] 
            [ , label={[rotate=90, anchor=west]above:{\small $\matrix{A}\vector{\Phi}(t_{1,1,2})$}}, l=30pt] 
            [ , label={[rotate=90, anchor=west]above:{\small $\matrix{A}\vector{\Phi}(t_{1,1,1})$}}, l=30pt]
        ]
      ]
    \end{forest}

    \begin{flushleft}
      where
    \end{flushleft}

    \begin{forest}
      bigtree
      [, label=below:{\small $\vector{d}_k$}
        [, label=left:{\small $\matrix{A} \tensor{I}_s^5$}, s sep=12pt
            [ , label={above:{\small $\vector{q}$}}, l=30pt]
            [ , label={[rotate=90, anchor=west]above:{\small $\matrix{A}\vector{\Phi}(t_{1,1,3})$}}, l=30pt] 
            [ , label={[rotate=90, anchor=west]above:{\small $\matrix{A}\vector{\Phi}(t_{1,1,2})$}}, l=30pt] 
            [ , label={[rotate=90, anchor=west]above:{\small $\matrix{A}\vector{\Phi}(t_{1,1,1})$}}, l=30pt]
        ]
      ]
    \end{forest}
    $\quad \xlongequal{ \vector{d}_{k+1} = \vector{d}_k \mathop{\times^1} \matrix{A} } \quad $
    \begin{forest}
      bigtree
      [, label=below:{\small $\vector{d}_{k+1} \mathop{\times^1} \tensor{I}_s^5$}
            [ , label={above:{\small $\vector{q}$}}, l=30pt]
            [ , label={[rotate=90, anchor=west]above:{\small $\matrix{A}\vector{\Phi}(t_{1,1,3})$}}, l=30pt] 
            [ , label={[rotate=90, anchor=west]above:{\small $\matrix{A}\vector{\Phi}(t_{1,1,2})$}}, l=30pt] 
            [ , label={[rotate=90, anchor=west]above:{\small $\matrix{A}\vector{\Phi}(t_{1,1,1})$}}, l=30pt]
      ]
    \end{forest}
    \label{fig:general-case2.2}
    \caption{The pivot residual condition $PR(n)$ is used to show that the remainder term in case 2.2 vanishes. The term is split into two terms by the definition of $W_n$, where the first term vanishes by the $PR(n)$ condition and the second term is reduced to a similar tree with a new root and smaller height.}
\end{figure}

\begin{remark}
    The PR condition is a technical condition that is only needed for the proof, and it is not clear whether it is necessary for the order conditions. It would be interesting to investigate whether this condition can be removed or replaced by a more natural one.

    By our observation in the analysis of various high order ERK schemes, the PR condition is often satisfied in a specific way by letting $W_n$ be the zero space, for example in \Cref{ex:cooperverner8}. Even more specifically, the $QD(m,n)$ condition already guarantees $PR(n)$ when $m$ and $n$ satisfy the requirement in the theorem, which makes \Cref{thm:my_sufficiency_thm} a special case of \Cref{thm:my_sufficiency_thm_general}.
\end{remark}


\section{Examples for checking sufficient conditions}\label{Examples}

We consider several well-known ERK schemes to check whether these newly proposed sufficient conditions are satisfied.


\begin{example}[classical RK4 method]
The classical RK4 method is one of the most famous example of the RK family, whose Butcher tableau is given below.
\begin{equation*}
    \begin{array}[b]{c|cccc}
    0   &            &            &            &            \\
    \frac{1}{2} & \frac{1}{2} &            &            &            \\
    \frac{1}{2} & 0          & \frac{1}{2} &            &            \\
    1           & 0          & 0           & 1          &            \\
    \hline
                & \frac{1}{6} & \frac{1}{3} & \frac{1}{3} & \frac{1}{6}
    \end{array}.
\end{equation*}
This method satisfies the conditions of \Cref{thm:my_sufficiency_thm} till $Q_1$ and $D_2$, and is a special case of our construction presented in this article when $m=1$, $n=2$, $p=4$. The theorem guarantees that it has order 4.

\end{example}

\begin{example}[Nystr\"om's method of order 5]
Nystr\"om's method is a correction of the one originally proposed by Kutta \cite{butcher1996history}. The method has the following tableau
\begin{equation*}
    \begin{array}[b]{c|cccccc}
    0   & 0        & 0        & 0        & 0        & 0        & 0 \\
    \frac{1}{3} & \frac{1}{3} & 0        & 0        & 0        & 0        & 0 \\
    \frac{2}{5} & \frac{4}{25} & \frac{6}{25} & 0        & 0        & 0        & 0 \\
    1   & \frac{1}{4} & -3       & \frac{15}{4} & 0        & 0        & 0 \\
    \frac{2}{3} & \frac{2}{27} & \frac{10}{9} & -\frac{50}{81} & \frac{8}{81} & 0 & 0 \\
    \frac{4}{5} & \frac{2}{25} & \frac{12}{25} & \frac{2}{15} & \frac{8}{75} & 0 & 0 \\[6pt]
    \hline
      & \frac{23}{192} & 0 & \frac{125}{192} & 0 & -\frac{27}{64} & \frac{125}{192}
    \end{array}.
\end{equation*}
We can compute the $\vector{d}$ and $\vector{q}$-vectors of this tableau. 
\begin{align*}
\vector{d}_0 &= 
\begin{bmatrix}
\dfrac{1}{192} & 0 & -\dfrac{25}{576} & \dfrac{1}{36} & \dfrac{9}{64} & -\dfrac{25}{192}
\end{bmatrix}^{\top}, \\
\vector{d}_1 &= 
\begin{bmatrix}
\dfrac{1}{384} & 0 & -\dfrac{35}{1152} & \dfrac{1}{36} & \dfrac{15}{128} & -\dfrac{15}{128}
\end{bmatrix}^{\top}, \\
\vector{q}_0 &= 
\begin{bmatrix}
0 & 0 & 0 & 0 & 0 & 0
\end{bmatrix}^{\top}, \\
\vector{q}_1 &= 
\begin{bmatrix}
0 & -\dfrac{1}{18} & 0 & 0 & 0 & 0
\end{bmatrix}^{\top}.
\end{align*}
We know that $D_1 = \Span\{ \vector{d}_0 \}$, $D_2 = \Span\{ \vector{d}_0, \vector{d}_1, \vector{d}_0 \times^1 \matrix{A}, \vector{d}_0 \odot \vector{c} \}$ and $Q_1 = \Span\{ \vector{q}_0 \} = \{ \vector{0} \}$, $Q_2 = \Span\{ \vector{q}_1 \}$, where $\vector{d}_0 \matrix{A}$ and $\vector{d}_0 \odot \vector{c}$ are computed to be
\begin{align*}
\vector{d}_0 \times^1 \matrix{A} &=
\begin{bmatrix}
0 & 0 & 0 & 0 & 0 & 0
\end{bmatrix}^{\top}, \\
\vector{d}_0 \odot \vector{c} &=
\begin{bmatrix}
0 & 0 & -\dfrac{5}{288} & \dfrac{1}{36} & \dfrac{3}{32} & -\dfrac{5}{48}
\end{bmatrix}^{\top}.
\end{align*}
One can check that they satisfy all conditions 1-5 of \Cref{thm:my_sufficiency_thm} till $Q_2$ and $D_2$, the theorem guarantees that it has order 5.

\end{example}

    

\begin{example}[Cooper and Verner's $8^\text{th}$ order method]\label{ex:cooperverner8}
Cooper and Verner constructed the 11-stage, 8-th order method in 1972 \cite{cooper1972some}. Its Butcher tableau is as follows:
\begin{center}
\begin{adjustbox}{max width=\columnwidth}
$
    \begin{array}[b]{c|c@{}c@{}c@{}c@{}c@{}c@{}c@{}c@{}c@{}c@{}c@{}}
    0 & 0 & 0 & 0 & 0 & 0 & 0 & 0 & 0 & 0 & 0 & 0 \\
    \frac{1}{2} & \frac{1}{2} & 0 & 0 & 0 & 0 & 0 & 0 & 0 & 0 & 0 & 0 \\
    \frac{1}{2} & \frac{1}{4} & \frac{1}{4} & 0 & 0 & 0 & 0 & 0 & 0 & 0 & 0 & 0 \\
    \frac{7+\sqrt{21}}{14} & \frac{1}{7} & \frac{-7-3\sqrt{21}}{98} & \frac{21+5\sqrt{21}}{49} & 0 & 0 & 0 & 0 & 0 & 0 & 0 & 0 \\
    \frac{7+\sqrt{21}}{14} & \frac{11+\sqrt{21}}{84} & 0 & \frac{18+4\sqrt{21}}{63} & \frac{21-\sqrt{21}}{252} & 0 & 0 & 0 & 0 & 0 & 0 & 0 \\
    \frac{1}{2} & \frac{5+\sqrt{21}}{48} & 0 & \frac{9+\sqrt{21}}{36} & \frac{-231+14\sqrt{21}}{360} & \frac{63-7\sqrt{21}}{80} & 0 & 0 & 0 & 0 & 0 & 0 \\
    \frac{7-\sqrt{21}}{14} & \frac{10-\sqrt{21}}{42} & 0 & \frac{-432+92\sqrt{21}}{315} & \frac{633-145\sqrt{21}}{90} & \frac{-504+115\sqrt{21}}{70} & \frac{63-13\sqrt{21}}{35} & 0 & 0 & 0 & 0 & 0 \\
    \frac{7+\sqrt{21}}{14} & \frac{1}{14} & 0 & 0 & 0 & \frac{14-3\sqrt{21}}{126} & \frac{13-3\sqrt{21}}{63} & \frac{1}{9} & 0 & 0 & 0 & 0 \\
    \frac{1}{2} & \frac{1}{32} & 0 & 0 & 0 & \frac{91-21\sqrt{21}}{576} & \frac{11}{72} & \frac{-385-75\sqrt{21}}{1152} & \frac{63+13\sqrt{21}}{128} & 0 & 0 & 0 \\
    \frac{7+\sqrt{21}}{14} & \frac{1}{14} & 0 & 0 & 0 & \frac{1}{9} & \frac{-733-147\sqrt{21}}{2205} & \frac{515+111\sqrt{21}}{504} & \frac{-51-11\sqrt{21}}{56} & \frac{132+28\sqrt{21}}{245} & 0 & 0 \\
    1 & 0 & 0 & 0 & 0 & \frac{-42+7\sqrt{21}}{18} & \frac{-18+28\sqrt{21}}{45} & \frac{-273-53\sqrt{21}}{72} & \frac{301+53\sqrt{21}}{72} & \frac{28-28\sqrt{21}}{45} & \frac{49-7\sqrt{21}}{18} & 0 \\[6pt]
    \hline
     & \frac{1}{20} & 0 & 0 & 0 & 0 & 0 & 0 & \frac{49}{180} & \frac{16}{45} & \frac{49}{180} & \frac{1}{20} 
    \end{array}.
$
\end{adjustbox}
\end{center}

Compute the $\vector{d}$ and $\vector{q}$-vectors of this tableau according to the definitions:

{\thinmuskip=1mu \medmuskip=1mu \thickmuskip=2mu
\begin{align*}
\vector{d}_0 &= 
\left[
\begin{array}{*{11}{>{\scriptstyle}c@{\hspace{2mm}}}}
0 & 0 & 0 & 0 & 0 & 0 & 0 & 0 & 0 & 0 & 0
\end{array}
\right]^{\top}, \\

\vector{d}_1 &= 
\left[
\begin{array}{*{11}{>{\scriptstyle}c@{\hspace{2mm}}}}
0 & 0 & 0 & 0 & -\frac{7}{144} + \frac{7 \sqrt{21}}{720} & -\frac{4}{225} + \frac{8 \sqrt{21}}{1575} & -\frac{7}{1440} - \frac{\sqrt{21}}{1440} & \frac{\sqrt{21}}{1440} + \frac{7}{1440} & \frac{4}{225} - \frac{8 \sqrt{21}}{1575} & \frac{7}{144} - \frac{7 \sqrt{21}}{720} & 0
\end{array}
\right]^{\top}, \\

\vector{d}_2 &= 
\left[
\begin{array}{*{11}{>{\scriptstyle}c@{\hspace{2mm}}}}
0 & 0 & 0 & 0 & -\frac{77}{1080} + \frac{\sqrt{21}}{72} & -\frac{5}{189} + \frac{\sqrt{21}}{105} & -\frac{83}{4320} - \frac{\sqrt{21}}{288} & \frac{\sqrt{21}}{288} + \frac{83}{4320} & \frac{5}{189} - \frac{\sqrt{21}}{105} & \frac{77}{1080} - \frac{\sqrt{21}}{72} & 0
\end{array}
\right]^{\top}, \\

\vector{d}_3 &= 
\left[
\begin{array}{*{11}{>{\scriptstyle}c@{\hspace{2mm}}}}
0 & 0 & 0 & 0 & -\frac{121}{1440} + \frac{23 \sqrt{21}}{1440} & -\frac{187}{6300} + \frac{146 \sqrt{21}}{11025} & -\frac{113}{2880} - \frac{149 \sqrt{21}}{20160} & \frac{149 \sqrt{21}}{20160} + \frac{113}{2880} & \frac{187}{6300} - \frac{146 \sqrt{21}}{11025} & \frac{121}{1440} - \frac{23 \sqrt{21}}{1440} & 0
\end{array}
\right]^{\top}, \\

\vector{q}_0 &= 
\left[
\begin{array}{*{11}{>{\scriptstyle}c@{\hspace{2mm}}}}
0 & 0 & 0 & 0 & 0 & 0 & 0 & 0 & 0 & 0 & 0
\end{array}
\right]^{\top}, \\

\vector{q}_1 &= 
\left[
\begin{array}{*{11}{>{\scriptstyle}c@{\hspace{2mm}}}}
0 & -\frac{1}{8} & 0 & 0 & 0 & 0 & 0 & 0 & 0 & 0 & 0
\end{array}
\right]^{\top}, \\

\vector{q}_2 &= 
\left[
\begin{array}{*{11}{>{\scriptstyle}c@{\hspace{2mm}}}}
0 & -\frac{1}{24} & \frac{1}{48} & -\frac{\sqrt{21}}{392} - \frac{1}{168} & 0 & 0 & 0 & 0 & 0 & 0 & 0
\end{array}
\right]^{\top}.
\end{align*}
}

Thus, the $D$-type spaces are $D_1 = \Span\{ \vector{d}_0 \} = {\vector{0}}$, $D_2 = \Span\{ \vector{d}_1 \}$, $D_3 = \Span\{ \vector{d}_1, \vector{d}_2, \vector{d}_1 \times^1 \matrix{A}, \vector{d}_1 \odot \vector{c} \}$, where

{\thinmuskip=1mu \medmuskip=1mu \thickmuskip=2mu
\begin{align*}
\vector{d}_1 \times^1 \matrix{A} &=
\left[
\begin{array}{*{11}{>{\scriptstyle}c@{\hspace{2mm}}}}
0 &
0 &
0 &
0 &
0 &
-\frac{1}{350} - \frac{\sqrt{21}}{3150} &
\frac{1}{160} + \frac{\sqrt{21}}{720} &
-\frac{1}{160} - \frac{\sqrt{21}}{720} &
\frac{1}{350} + \frac{\sqrt{21}}{3150} &
0 &
0
\end{array}
\right]^{\top},
\\
\vector{d}_1 \odot \vector{c} &=
\left[
\begin{array}{*{11}{>{\scriptstyle}c@{\hspace{2mm}}}}
0 &
0 &
0 &
0 &
-\frac{7}{720} + \frac{\sqrt{21}}{720} &
-\frac{2}{225} + \frac{4\sqrt{21}}{1575} &
-\frac{1}{720} &
\frac{1}{720} &
\frac{2}{225} - \frac{4\sqrt{21}}{1575} &
\frac{7}{720} - \frac{\sqrt{21}}{720} &
0
\end{array}
\right]^{\top}.
\end{align*}
}
Since the four vectors in $D_3$ only span a 3-dimensional space, $D_3$ can be written as $D_3 = \Span\{ \vector{d}_1, \vector{d}_2, \vector{d}_1 \odot \vector{c} \}$. Then, $D_4 = \Span\{ D_3, \vector{d}_2 \odot \vector{c}, \vector{d}_1 \odot \vector{c}^2, \vector{d}_2 \times^1 \matrix{A}, (\vector{d}_1 \odot \vector{c}) \times^1 \matrix{A} \}$, where
\resizebox{\textwidth}{!}{
\begin{minipage}{\textwidth}
{\thinmuskip=1mu \medmuskip=1mu \thickmuskip=2mu
\begin{align*}
\vector{d}_2 \times^1 \matrix{A} &=
\left[
\begin{array}{*{11}{>{\scriptstyle}c@{\hspace{2mm}}}}
0 &
0 &
\frac{3}{1400} + \frac{\sqrt{21}}{4200} &
\frac{\sqrt{21}}{1200} &
-\frac{\sqrt{21}}{1200} &
-\frac{13}{1800} - \frac{239\sqrt{21}}{264600} &
\frac{43}{2880} + \frac{29\sqrt{21}}{8640} &
-\frac{43}{2880} - \frac{29\sqrt{21}}{8640} &
\frac{8}{1575} + \frac{22\sqrt{21}}{33075} &
0 &
0
\end{array}
\right]^{\top}, \\
\vector{d}_2 \odot \vector{c} &=
\left[
\begin{array}{*{11}{>{\scriptstyle}c@{\hspace{2mm}}}}
0 &
0 &
0 &
0 &
-\frac{2}{135} + \frac{\sqrt{21}}{540} &
-\frac{5}{378} + \frac{\sqrt{21}}{210} &
-\frac{19}{4320} - \frac{11\sqrt{21}}{30240} &
\frac{19}{4320} + \frac{11\sqrt{21}}{30240} &
\frac{5}{378} - \frac{\sqrt{21}}{210} &
\frac{2}{135} - \frac{\sqrt{21}}{540} &
0
\end{array}
\right]^{\top}, \\
(\vector{d}_1 \odot \vector{c}) \times^1 \matrix{A} &=
\left[
\begin{array}{*{11}{>{\scriptstyle}c@{\hspace{2mm}}}}
0 &
0 &
\frac{1}{4200} - \frac{\sqrt{21}}{4200} &
-\frac{7}{2400} + \frac{\sqrt{21}}{2400} &
\frac{7}{2400} - \frac{\sqrt{21}}{2400} &
-\frac{3}{1400} - \frac{11\sqrt{21}}{88200} &
\frac{1}{240} + \frac{\sqrt{21}}{1008} &
-\frac{1}{240} - \frac{\sqrt{21}}{1008} &
\frac{1}{525} + \frac{4\sqrt{21}}{11025} &
0 &
0
\end{array}
\right]^{\top},
\\
\vector{d}_1 \odot \vector{c}^{2} &=
\left[
\begin{array}{*{11}{>{\scriptstyle}c@{\hspace{2mm}}}}
0 &
0 &
0 &
0 &
-\frac{1}{360} &
-\frac{1}{225} + \frac{2\sqrt{21}}{1575} &
-\frac{1}{1440} + \frac{\sqrt{21}}{10080} &
\frac{1}{1440} - \frac{\sqrt{21}}{10080} &
\frac{1}{225} - \frac{2\sqrt{21}}{1575} &
\frac{1}{360} &
0
\end{array}
\right]^{\top}.
\end{align*}
}
\vspace{0em}
\end{minipage}
}

The $Q$-type spaces are given by $Q_1 = \Span\{ \vector{q}_0 \} = \{ \vector{0} \}$, $Q_2 = \Span \{ \vector{q}_1 \}$, $Q_3 = \Span \{ \vector{q}_1, \vector{q}_2, \allowbreak \matrix{A} \vector{q}_1 \}$, where
{\thinmuskip=1mu \medmuskip=1mu \thickmuskip=2mu
\begin{align*}
\matrix{A} \vector{q}_1 &= 
\left[
\begin{array}{*{11}{>{\scriptstyle}c@{\hspace{2mm}}}}
0 &
0 &
-\frac{1}{32} &
\frac{1}{112} + \frac{3\sqrt{21}}{784} &
0 &
0 &
0 &
0 &
0 &
0 &
0
\end{array}
\right]^{\top}.
\end{align*}
}

One can check that they satisfy all conditions 1-5 of \Cref{thm:my_sufficiency_thm} till $Q_3$ and $D_4$, and the theorem guarantees that it has order 8.

\paragraph{Discussion}

All of the above examples satisfy the assumptions of both sufficiency theorems. It is nontrivial to determine whether there exist (explicit) RK schemes that do not satisfy all the conditions of either sufficiency theorem, yet still attain the desired order of accuracy. Although the two sufficiency theorems are clearly not equivalent from the perspective of their construction, it is also nontrivial—albeit seemingly easier—to construct an ERK scheme whose order can be established by one theorem but not by the other.

We give the following conjecture on the necessity of the conditions in the sufficiency theorem:



\begin{conjecture}
The set of conditions in \Cref{thm:my_sufficiency_thm_general} is not only sufficent but also necessary for a vector-valued explicit/implicit RK method to achieve order $p$.
\end{conjecture}


Proving or giving a counterexample of the conjecture will yield significant impact on the theoretical foundation of simplifying assumptions for explicit RK methods, and further innovate other ideas for construction of related schemes. Unforturenately, this goal is far from accomplished. We don't have a universal way to test all possible methods, and up to now all methods we have tested manually satisfy all conditions for the sufficiency theorem for some $m,n$, which successfully give the correct order.

\end{example}

\section{A construction example for $p=10$}

    We illustrate the stage arrangement for $p = 10$ in Figure~\ref{fig:stage-arrangement-p10}, where $m=4$, $n=5$, $N=6$, $l=6$, $s=22$. A $6$-point Lobatto quadrature rule with nodes $\{ x_i \}_{i=1}^6$ and weights $\{ w_i \}_{i=1}^6$ is used to construct the stage vector $\vector{c}$ and weights $\vector{b}$. The stage vector $\vector{c}$ is shown vertically on the left of the tableau and the weight vector $\vector{b}$ is shown horizontally at the bottom of the tableau. The stages $c_i$, with $i=2,\, \dots,\, 7$, are free parameters, and the weights satisfy
    \begin{align*}
        w_{44} + w_{54} = w_2, \\
        w_{33} + w_{43} + w_{53} = w_3, \\
        w_{22} + w_{32} + w_{42} + w_{52} = w_4, \\
        w_{11} + w_{21} + w_{31} + w_{41} + w_{51} = w_{5}.
    \end{align*}

    \begin{figure}[!ht]
        \centering

\begin{tikzpicture}[scale=0.9, every node/.style={font=\small}]
    \def\cellsize{0.6} 
    \def\n{22} 
    \def\l{7} 

    \foreach \i in {0,...,\n} {
        \draw[gray!50] (0, \i*\cellsize) -- (\n*\cellsize, \i*\cellsize);
        \draw[gray!50] (\i*\cellsize, 0) -- (\i*\cellsize, \n*\cellsize);
    }

    \draw[black] (-\cellsize,0) -- ({\n*\cellsize}, 0);
    \draw[black] (0, -1*\cellsize) -- (0, \n*\cellsize);

    \def\dvlabels{$x_5$,$x_4$,$x_3$,$x_2$,  $x_5$,$x_4$,$x_3$,$x_2$,  $x_5$,$x_4$,$x_3$,  $x_5$,$x_4$,  $x_5$,  $x_6=1$ }
    \def\qvlabels{\contour{white}{$c_1=0$},$c_2$,$c_3$,$c_4$,$c_5$,$c_6$,$c_7$}
    \def\dhlabels{$w_{51}$,$w_{52}$,$w_{53}$,$w_{54}$,  $w_{41}$,$w_{42}$,$w_{43}$,$w_{44}$,  $w_{31}$,$w_{32}$,$w_{33}$,  $w_{21}$,$w_{22}$,  $w_{11}$,  $w_6$ }
    \def\qhlabels{$w_1$,$0$,$0$,$0$,$0$,$0$,$0$}

    \foreach \lab [count=\i] in \dvlabels {
    \node [anchor=east] (dv\i) at ({-0.2*\cellsize},{(\n-\l-\i+0.5)*\cellsize}) {\lab};
    }
    \foreach \lab [count=\i] in \qvlabels {
    \node [anchor=east] (qv\i) at ({-0.2*\cellsize},{(\n-\i+0.5)*\cellsize}) {\lab};
    }
    
    \foreach \lab [count=\i] in \dhlabels {
    \node (dh\i) at ({(\l-0.5+\i)*\cellsize}, -0.5*\cellsize) {\lab};
    }
    \foreach \lab [count=\i] in \qhlabels {
    \node (qh\i) at ({(-0.5+\i)*\cellsize}, -0.5*\cellsize) {\lab};
    }

    
    \draw [decorate, decoration={brace, amplitude=5pt, mirror}] 
        ({-1.2*\cellsize}, {(\n-7+6)*\cellsize}) -- ({-1.2*\cellsize}, {(\n-2)*\cellsize})
        node [midway, left=6pt] {$Q$-group 1};
    \draw [decorate, decoration={brace, amplitude=5pt, mirror}] 
        ({-1.2*\cellsize}, {(\n-2)*\cellsize}) -- ({-1.2*\cellsize}, {(\n-4)*\cellsize})
        node [midway, left=6pt] {$Q$-group 2};
    \draw [decorate, decoration={brace, amplitude=5pt, mirror}] 
        ({-1.2*\cellsize}, {(\n-4)*\cellsize}) -- ({-1.2*\cellsize}, {(\n-7)*\cellsize})
        node [midway, left=6pt] {$Q$-group 3};

    \draw [decorate, decoration={brace, amplitude=5pt, mirror}] 
        ({-1.2*\cellsize}, {(\n-8+1)*\cellsize}) -- ({-1.2*\cellsize}, {(\n-11)*\cellsize})
        node [midway, left=6pt] {$D$-group 5};
    \draw [decorate, decoration={brace, amplitude=5pt, mirror}] 
        ({-1.2*\cellsize}, {(\n-12+1)*\cellsize}) -- ({-1.2*\cellsize}, {(\n-15)*\cellsize})
        node [midway, left=6pt] {$D$-group 4};
    \draw [decorate, decoration={brace, amplitude=5pt, mirror}] 
        ({-1.2*\cellsize}, {(\n-16+1)*\cellsize}) -- ({-1.2*\cellsize}, {(\n-18)*\cellsize})
        node [midway, left=6pt] {$D$-group 3};
    \draw [decorate, decoration={brace, amplitude=5pt, mirror}] 
        ({-1.2*\cellsize}, {(\n-19+1)*\cellsize}) -- ({-1.2*\cellsize}, {(\n-20)*\cellsize})
        node [midway, left=6pt] {$D$-group 2};
    \draw [decorate, decoration={brace, amplitude=5pt, mirror}] 
        ({-1.2*\cellsize}, {(\n-21+1)*\cellsize}) -- ({-1.2*\cellsize}, {(\n-21)*\cellsize})
        node [midway, left=6pt] {$D$-Group 1};

    \tikzset{
        myrect/.style={
            draw=black, thick, rounded corners=3pt,
            fill=white, opacity=0.8, inner sep=2pt
        },
        mypoly/.style={
            draw=black, thick, rounded corners=3pt,
            fill=white, opacity=0.8, inner sep=2pt
        },
        shaded/.style={
            draw=none,
            pattern=north east lines,
            pattern color=gray
        }
    }

    \draw[mypoly] (0, 0) -- (0, \n*\cellsize) -- (\l*\cellsize, {(\n-\l)*\cellsize}) -- (\l*\cellsize, 0) -- cycle;
    \node at ({3.5*\cellsize}, {(\n/2-3.5)*\cellsize}) {$Q$-region};
    \draw[mypoly] (\l*\cellsize, 0) -- (\l*\cellsize, {(\n-\l)*\cellsize}) -- (\n*\cellsize, 0) -- cycle;
    \node at ({(\l+4.5)*\cellsize}, {4.5*\cellsize}) {$D$-region};

    \foreach \i in {1,...,\n} {
        \node at (\i*\cellsize - \cellsize/2, \n*\cellsize - \i*\cellsize + \cellsize/2) {\contour{white}{$0$}};
    }

\end{tikzpicture}
        \caption{Stage arrangement for $p = 10$. The two regions, $D$-region and $Q$-region in the figure, correspond to the parameters in matrix $\matrix{A}$ that will be utilized to construct $D$-type spaces and $Q$-type spaces, respectively. }
        \label{fig:stage-arrangement-p10}
    \end{figure}
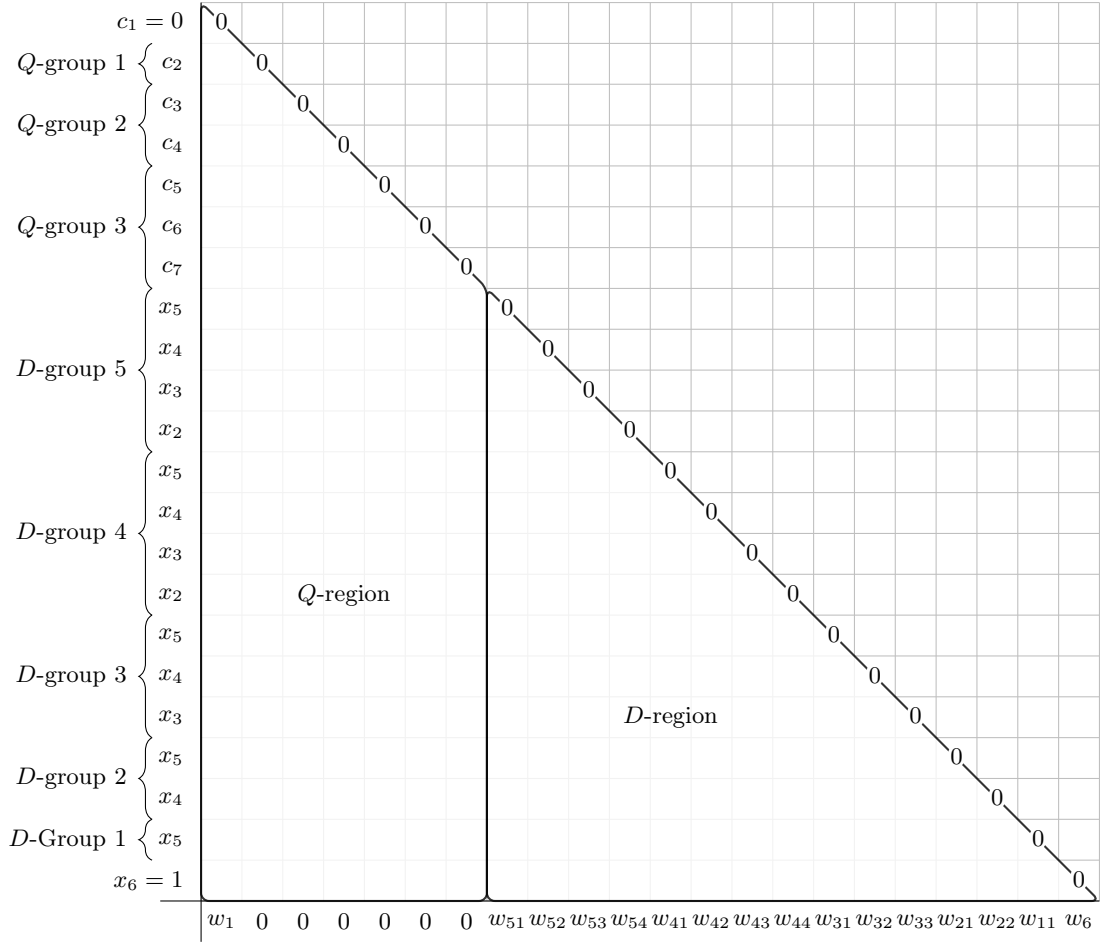

    Consider the construction of ERK methods of order $p=10$. The stage arrangement and quadrature initialization is displayed in Example 4.1 of \cite{he2026low} and \Cref{fig:stage-arrangement-p10}. According to the theoretical framework established in this section, achieving order $p$ requires constructing the $D$-type spaces up to level $D_{p/2}$. For $p=10$, this corresponds to $D_5$. We next illustrate how the corresponding \(D\)-system is constructed in this setting, demonstrating how the abstract theory translates into concrete computational procedures.

    %


    As shown in \Cref{fig:tikz_d_conditions_D5_p10}, the variable index sets for $D$-system with $p=10$ are:
    \begin{align*}
        I_1^D &= \big\{ (22,j) : 8 \leq j \leq 21 \big\}, \\
        I_2^D &= \big\{ (21,j) : 8 \leq j \leq 20 \big\}, \\
        I_3^D &= \big\{ (i,j) : i \in \{19,20\},\, 8 \leq j \leq 18 \big\}, \\
        I_4^D &= \big\{ (i,j) : i \in \{16,17,18\},\, 8 \leq j \leq 15 \big\}, \\
        I_5^D &= \big\{ (i,j) : i \in \{12,13,14,15\},\, 8 \leq j \leq 11 \big\}.
    \end{align*}
    The total number of variables in the $D$-system is $\sum_{k=1}^5 |I_k^D| = 14 + 13 + 2 \times 11 + 3 \times 8 + 4 \times 4 = 89$.
     As shown in \Cref{fig:tikz_d_conditions_D5_p10}, there remain 16 entries in the $D$-region of $\matrix{A}$ that are not determined by the \(D\)-system and can therefore be treated as free variables. In our construction, these entries are initialized to zero. 

    For each level $k = 1, \dots, 5$, the cluster space $W^D_k$ has dimension $k(k-1)/2$. The orthogonal complement $(\tilde{W}^D_k)^\perp \cap V_l^0$ has dimension $(s-\frac{k(k-1)}{2})-(l+2) = (22-\frac{k(k-1)}{2})-8 = 14-\frac{k(k-1)}{2}$. Consequently, the number of equations at level $k$ is:
    \begin{equation*}
        N_k^{\text{eq}} = 
        \begin{cases}
            14, & \text{for } k = 1, \\
            (k-1)\left(14-\frac{k(k-1)}{2}\right), & \text{for } k \geq 2.
        \end{cases}
    \end{equation*}
    This yields $N_1^{\text{eq}} = 14$, $N_2^{\text{eq}} = 13$, $N_3^{\text{eq}} = 22$, $N_4^{\text{eq}} = 24$, and $N_5^{\text{eq}} = 16$. The total number of equations is $89$, which is exactly the same as the number of variables.



    \begin{table}[!ht]
        \centering
        \caption{The complete set of linear constraints imposed by the construction of $D_5$ for tenth-order ERK methods. Each row represents an equation of the form $\vector{d}_{k-1} \cdot \vector{v} = 0$ or $(\vector{u} \mathop{\times^1} \matrix{A}) \cdot \vector{v} = 0$, where $\vector{u} \in W_{k-1}^D \cap (W_{k-2}^D)^\perp$ and $\vector{v} \in (\tilde{W}_k^D)^\perp \cap V_l^0$.}

\renewcommand{\arraystretch}{1.5}
\begin{tabular}{c|c}
\toprule
$k$ & equations  \\ \midrule
$1$ & $ \vector{d}_0 \cdot \vector{v} = 0, \quad \forall \vector{v} \in (\tilde{W}_1^D)^\perp \cap V_l^0 $ \\ \midrule
$2$ & \makecell{$\begin{aligned} \vector{d}_{1} \cdot \vector{v} &= 0, \quad \forall \vector{v} \in (\tilde{W}_2^D)^\perp \cap V_l^0 \\ (\vector{u} {\times^1} \matrix{A}) \cdot \vector{v} &= 0, \quad \forall \vector{u} \in W_{1}^D,\, \forall \vector{v} \in (\tilde{W}_2^D)^\perp \cap V_l^0 \end{aligned}$} \\ \midrule
$3$ & \makecell{$\begin{aligned} \vector{d}_{2} \cdot \vector{v} &= 0, \quad \forall \vector{v} \in (\tilde{W}_3^D)^\perp \cap V_l^0 \\ (\vector{u} {\times^1} \matrix{A}) \cdot \vector{v} &= 0, \quad \forall \vector{u} \in W_{2}^D \cap (W_{1}^D)^\perp,\, \forall \vector{v} \in (\tilde{W}_3^D)^\perp \cap V_l^0 \end{aligned} $} \\ \midrule
$4$ & \makecell{$\begin{aligned} \vector{d}_{3} \cdot \vector{v} &= 0, \quad \forall \vector{v} \in (\tilde{W}_4^D)^\perp \cap V_l^0 \\ (\vector{u} {\times^1} \matrix{A}) \cdot \vector{v} &= 0, \quad \forall \vector{u} \in W_{3}^D \cap (W_{2}^D)^\perp,\, \forall \vector{v} \in (\tilde{W}_4^D)^\perp \cap V_l^0 \end{aligned}$} \\ 
\bottomrule
\end{tabular}

        \label{tab:table_d_conditions_D5_p10}
    \end{table}

    \begin{figure}[!ht]
        \centering

\begin{tikzpicture}[scale=1.0, every node/.style={font=\small}]
    \def\cellsize{0.6} 
    \def\n{22} 
    \def\l{7} 

    \foreach \i in {0,...,\n} {
        \draw[gray!50] (0, \i*\cellsize) -- (\n*\cellsize, \i*\cellsize);
        \draw[gray!50] (\i*\cellsize, 0) -- (\i*\cellsize, \n*\cellsize);
    }

    \draw[black] (0,0) -- ({(\n+1)*\cellsize}, 0);
    \draw[black] (\n*\cellsize, -1*\cellsize) -- (\n*\cellsize, \n*\cellsize);

    \def\labels{$x_5$,$x_4$,$x_3$,$x_2$,  $x_5$,$x_4$,$x_3$,$x_2$,  $x_5$,$x_4$,$x_3$,  $x_5$,$x_4$,  $x_5$,  $x_6$ }

    \foreach \lab [count=\i] in \labels {
    \node (v\i) at ({(\n+0.5)*\cellsize},{(\n-\l-\i+0.5)*\cellsize}) {\lab};
    }
    
    \foreach \lab [count=\i] in \labels {
    \node (h\i) at ({(\l-0.5+\i)*\cellsize}, -0.5*\cellsize) {\lab};
    }

    \tikzset{
        myrect/.style={
            draw=black, thick, rounded corners=3pt,
            fill=white, opacity=0.8, inner sep=2pt
        },
        mypoly/.style={
            draw=black, ultra thick, dashed, rounded corners=3pt,
            fill=white, opacity=0.8, inner sep=2pt
        },
        shaded/.style={
            draw=none,
            pattern=north east lines,
            pattern color=gray
        }
    }

    \draw[mypoly] (\l*\cellsize, 0) -- (\l*\cellsize, {(\n-\l)*\cellsize}) -- (\n*\cellsize, 0) -- cycle;

    \foreach \i in {1,...,\n} {
        \node at (\i*\cellsize - \cellsize/2, \n*\cellsize - \i*\cellsize + \cellsize/2) {\contour{white}{$0$}};
    }

    \draw[myrect] (\l*\cellsize, 0) rectangle ++({(\n-1-\l)*\cellsize}, \cellsize);
    \node at ({\l*\cellsize + (\n-1-\l)/2*\cellsize}, 0.5*\cellsize) {$\vector{d}_0$};

    \begin{scope}
        \draw[myrect] ({(\l)*\cellsize}, \cellsize) rectangle ++({((\n-4)-(\l))*\cellsize}, \cellsize);
        \draw[shaded] ({(\l+2)*\cellsize}, \cellsize) rectangle ++({((\n-4)-(\l+2))*\cellsize}, \cellsize);
        \node at ({(\l)*\cellsize + ((\n-4)-(\l))/2*\cellsize}, 1.5*\cellsize) {\contour{white}{$\vector{\varepsilon}_{11}\matrix{A}$}};
    \end{scope}

    \begin{scope}
        \draw[myrect] ({(\l)*\cellsize}, 2*\cellsize) rectangle ++({((\n-7)-(\l))*\cellsize}, \cellsize);
        \draw[shaded] ({(\l+3)*\cellsize}, 2*\cellsize) rectangle ++({((\n-7)-(\l+3))*\cellsize}, \cellsize);
        \node at ({(\l)*\cellsize + ((\n-7)-(\l))/2*\cellsize}, 2.5*\cellsize) {\contour{white}{$\vector{\varepsilon}_{21} {\times^1} \matrix{A}$}};
    \end{scope}

    \begin{scope}
        \draw[myrect] ({(\l)*\cellsize}, 3*\cellsize) rectangle ++({((\n-7)-(\l))*\cellsize}, \cellsize);
        \draw[shaded] ({(\l+3)*\cellsize}, 3*\cellsize) rectangle ++({((\n-7)-(\l+3))*\cellsize}, \cellsize);
        \node at ({(\l)*\cellsize + ((\n-7)-(\l))/2*\cellsize}, 3.5*\cellsize) {\contour{white}{$\vector{\varepsilon}_{22} {\times^1} \matrix{A}$}};
    \end{scope}

    \begin{scope}
        \draw[myrect] ({(\l)*\cellsize}, 4*\cellsize) rectangle ++({((\n-11)-(\l))*\cellsize}, \cellsize);
        \draw[shaded] ({(\l+4)*\cellsize}, 4*\cellsize) rectangle ++({((\n-11)-(\l+4))*\cellsize}, \cellsize);
        \node at ({(\l)*\cellsize + ((\n-11)-(\l))/2*\cellsize}, 4.5*\cellsize) {\contour{white}{$\vector{\varepsilon}_{31} {\times^1} \matrix{A}$}};
    \end{scope}

    \begin{scope}
        \draw[myrect] ({(\l)*\cellsize}, 5*\cellsize) rectangle ++({((\n-11)-(\l))*\cellsize}, \cellsize);
        \draw[shaded] ({(\l+4)*\cellsize}, 5*\cellsize) rectangle ++({((\n-11)-(\l+4))*\cellsize}, \cellsize);
        \node at ({(\l)*\cellsize + ((\n-11)-(\l))/2*\cellsize}, 5.5*\cellsize) {\contour{white}{$\vector{\varepsilon}_{32} {\times^1} \matrix{A}$}};
    \end{scope}

    \begin{scope}
        \draw[myrect] ({(\l)*\cellsize}, 6*\cellsize) rectangle ++({((\n-11)-(\l))*\cellsize}, \cellsize);
        \draw[shaded] ({(\l+4)*\cellsize}, 6*\cellsize) rectangle ++({((\n-11)-(\l+4))*\cellsize}, \cellsize);
        \node at ({(\l)*\cellsize + ((\n-11)-(\l))/2*\cellsize}, 6.5*\cellsize) {\contour{white}{$\vector{\varepsilon}_{33} {\times^1} \matrix{A}$}};
    \end{scope}

    \begin{scope}
        \draw[myrect] ({(\n-4)*\cellsize}, \cellsize) rectangle ++(2*\cellsize, \cellsize);
        \node at ({(\n-4)*\cellsize + 1*\cellsize}, 1.5*\cellsize) {$\vector{d}_1$};

        \draw[myrect] ({(\n-7)*\cellsize}, 2*\cellsize) rectangle ++(3*\cellsize, \cellsize);
        \node at ({(\n-7)*\cellsize + 1.5*\cellsize}, 2.5*\cellsize) {$\vector{d}_1$};

        \draw[myrect] ({(\n-11)*\cellsize}, 4*\cellsize) rectangle ++(4*\cellsize, \cellsize);
        \node at ({(\n-11)*\cellsize + 2*\cellsize}, 4.5*\cellsize) {$\vector{d}_1$};

        \draw[myrect] ({(\l)*\cellsize}, 7*\cellsize) rectangle ++(4*\cellsize, \cellsize);
        \node at ({(\l)*\cellsize + 2*\cellsize}, 7.5*\cellsize) {$\vector{d}_1$};
    \end{scope}

    \begin{scope}
        \draw[myrect] ({(\n-7)*\cellsize}, 3*\cellsize) rectangle ++(3*\cellsize, \cellsize);
        \node at ({(\n-7)*\cellsize + 1.5*\cellsize}, 3.5*\cellsize) {$\vector{d}_2$};


        \draw[myrect] ({(\n-11)*\cellsize}, 5*\cellsize) rectangle ++(4*\cellsize, \cellsize);
        \node at ({(\n-11)*\cellsize + 2*\cellsize}, 5.5*\cellsize) {$\vector{d}_2$};


        \draw[myrect] ({(\l)*\cellsize}, 8*\cellsize) rectangle ++(4*\cellsize, \cellsize);
        \node at ({(\l)*\cellsize + 2*\cellsize}, 8.5*\cellsize) {$\vector{d}_2$};
    \end{scope}

    \begin{scope}
        \draw[myrect] ({(\n-11)*\cellsize}, 6*\cellsize) rectangle ++(4*\cellsize, \cellsize);
        \node at ({(\n-11)*\cellsize + 2*\cellsize}, 6.5*\cellsize) {$\vector{d}_3$};



        \draw[myrect] ({(\l)*\cellsize}, 9*\cellsize) rectangle ++(4*\cellsize, \cellsize);
        \node at ({(\l)*\cellsize + 2*\cellsize}, 9.5*\cellsize) {$\vector{d}_3$};
    \end{scope}

    \begin{scope}
        \draw[myrect] ({(\l)*\cellsize}, 10*\cellsize) rectangle ++(4*\cellsize, \cellsize);
        \node at ({(\l)*\cellsize + 2*\cellsize}, 10.5*\cellsize) {$\vector{d}_4$};
    \end{scope}

\end{tikzpicture}
        \caption{Spatial configuration of the $D$-system equations on the Butcher tableau $\matrix{A}$ for $p=10$. The matrix shows the distribution of different equation types: $\vector{d}_k$ regions and $\vector{\varepsilon}_{i,j} \mathop{\times^1} \matrix{A}$ regions. Blank areas represent the entries that are initialized to zero, and shaded areas indicate entries that become zero as a direct consequence of solving the constructed $D$-system equations. Dashed lines represent the boundary of the $D$-region.}
        \label{fig:tikz_d_conditions_D5_p10}
    \end{figure}

    \Cref{tab:table_d_conditions_D5_p10} provides a comprehensive listing of all constraints imposed by the construction of $D_5$. Each entry in the table corresponds to a specific linear equation of the form $\vector{d}_{k-1} \cdot \vector{v} = 0$ or $(\vector{u} \mathop{\times^1} \matrix{A}) \cdot \vector{v} = 0$, where $\vector{u}$ is a basis vector of $W_{k-1}^D \cap (W_{k-2}^D)^\perp$ and $\vector{v}$ is a basis vector of $(\tilde{W}_k^D)^\perp \cap V_l^0$.

    \Cref{fig:tikz_d_conditions_D5_p10} visualizes the spatial distribution of equations on the Butcher tableau $\matrix{A}$. The figure illustrates how different types of constraints ($\vector{d}_k$ residuals and $\vector{\varepsilon}_{i,j}\matrix{A}$ products) occupy distinct regions of the matrix when the linear system is assembled (we refer our readers to appendix if you are interested). The shaded regions indicate entries that are zero by construction, while the labeled rectangles show the blocks affected by each type of equation.

    Consider the construction of ERK methods of order $p=10$. The stage arrangement and quadrature initialization is displayed in Example 4.1 of \cite{he2026low} and \Cref{fig:stage-arrangement-p10}.
    The construction of ERK method of order $p=10$ in the $D$-region of table $\matrix{A}$ is given in
    $D$-systems. For $p=10$, we have $m=4$. We next illustrate how the
    corresponding $Q$-system is constructed in this setting, demonstrating how the abstract theory
    translates into concrete computational procedures.
    


    Applying the general formulation from $Q$-system, the variable index sets for $p=10$ are:
    \begin{align*}
        I_1^Q &= \big\{ (i,1) : 2 \leq i \leq 22  \big\}, \\
        I_2^Q &= \big\{ (i,2) : 3 \leq i \leq 22 \big\}, \\
        I_3^Q &= \big\{ (i,j) : 5 \leq i \leq 22,\, 3 \leq j \leq 4 \big\}, \\
        I_4^Q &= \big\{ (i,j) : 8 \leq i \leq 22,\, 5 \leq j \leq 7 \big\}.
    \end{align*}
    The total number of variables in the $Q$-system is $\sum_{k=1}^4 |I_k^Q| = 21 + 20 + 2 \times 18 + 3 \times 15 = 122$. As shown in \Cref{fig:tikz_q_conditions_Q4_p10}, there remain 4 entries in the $Q$-region of $\matrix{A} $ that are not determined by the Q-system and can therefore be treated as free variables. In our construction, these entries are initialized to zero.

    For each level $k = 1, \dots, 4$, the test space $(\tilde{W}_k^Q)^\perp$ has dimension $(s-\frac{k(k-1)}{2}-1) = 21-\frac{k(k-1)}{2}$. Consequently, the number of equations at level $k$ is:
    \begin{equation*}
        M_k^{\text{eq}} = 
        \begin{cases}
            21, & \text{for } k = 1, \\
            (k-1)\left(21-\frac{k(k-1)}{2}\right), & \text{for } k \geq 2,
        \end{cases}
    \end{equation*}
    yielding $M_1^{\text{eq}} = 21$, $M_2^{\text{eq}} = 20$, $M_3^{\text{eq}} = 36$, and $M_4^{\text{eq}} = 45$. The total number of equations is $\sum_{k=1}^4 M_k^{\text{eq}} = 122$, which is exactly the same as the number of variables.

    \Cref{tab:table_q_conditions_Q4_p10} provides a comprehensive listing of all constraints imposed by the construction of $Q_4$. Each entry in the table corresponds to a specific linear equation of the form $Q$-system. \Cref{fig:tikz_q_conditions_Q4_p10} visualizes the spatial distribution of equations on the Butcher tableau $\matrix{A}$ for $p=10$. The figure illustrates how different types of constraints ($\vector{q}_k$ residuals and $\matrix{A} \vector{e}_{j}$ products) occupy distinct regions of the matrix when the linear system is assembled. The $Q$-region is enclosed by the dashed line. Blank regions within the $Q$-region are entries initialized to zero, shaded regions indicate entries that are zero as a result of solving the $Q$-system equations, and labeled rectangles show the blocks affected by each type of equation. This design allows the equations to be assembled into a block diagonal linear system. The dashed lines clearly separate the $Q$-region from the $D$-region, highlighting the modular structure of the overall construction.

    \begin{table}[!ht]
        \centering
        \caption{The complete set of linear constraints imposed by the construction of $Q_4$ for tenth-order ERK methods.}



\renewcommand{\arraystretch}{1.5}
\begin{tabular}{c|c}
\toprule
$k$ & equations  \\ \midrule
$1$ & $ \vector{q}_0 \cdot \vector{v} = 0, \quad \forall \vector{v} \in (\tilde{W}_1^Q)^\perp $ \\ \midrule
$2$ & \makecell{$\begin{aligned} \vector{q}_{1} \cdot \vector{v} &= 0, \quad \forall \vector{v} \in (\tilde{W}_2^Q)^\perp \\ (\matrix{A} \vector{u}) \cdot \vector{v} &= 0, \quad \forall \vector{v} \in (\tilde{W}_2^Q)^\perp, \forall \vector{u} \in W_{1}^Q \end{aligned}$} \\ \midrule
$3$ & \makecell{$\begin{aligned} \vector{q}_{2} \cdot \vector{v} &= 0, \quad \forall \vector{v} \in (\tilde{W}_3^Q)^\perp \\ (\matrix{A} \vector{u}) \cdot \vector{v} &= 0, \quad \forall \vector{v} \in (\tilde{W}_3^Q)^\perp, \forall \vector{u} \in W_{2}^Q \cap (W_{1}^Q)^\perp \end{aligned} $} \\ \midrule
$4$ & \makecell{$\begin{aligned} \vector{q}_{3} \cdot \vector{v} &= 0, \quad \forall \vector{v} \in (\tilde{W}_4^Q)^\perp \\ (\matrix{A} \vector{u}) \cdot \vector{v} &= 0, \quad \forall \vector{v} \in (\tilde{W}_4^Q)^\perp, \forall \vector{u} \in W_{3}^Q \cap (W_{2}^Q)^\perp \end{aligned}$} \\ 
\bottomrule
\end{tabular}

        \label{tab:table_q_conditions_Q4_p10}
    \end{table}

    \begin{figure}[!ht]
        \centering

\begin{tikzpicture}[scale=1.0, every node/.style={font=\small}]
    \def\cellsize{0.6} 
    \def\n{22} 
    \def\l{7} 

    \foreach \i in {0,...,\n} {
        \draw[gray!50] (0, \i*\cellsize) -- (\n*\cellsize, \i*\cellsize);
        \draw[gray!50] (\i*\cellsize, 0) -- (\i*\cellsize, \n*\cellsize);
    }

    \tikzset{
        myrect/.style={
            draw=black, thick, rounded corners=3pt,
            fill=white, opacity=0.8, inner sep=2pt
        },
        mypoly/.style={
            draw=black, ultra thick, dashed, rounded corners=3pt,
            fill=white, opacity=0.8, inner sep=2pt
        },
        shaded/.style={
            draw=none,
            pattern=north east lines,
            pattern color=gray
        }
    }

    \draw[mypoly] (0, 0) -- (0, \n*\cellsize) -- (\l*\cellsize, {(\n-\l)*\cellsize}) -- (\l*\cellsize, 0) -- cycle;
    \node at ({3.5*\cellsize}, {(\n/2-3.5)*\cellsize}) {$Q$-region};

    \foreach \i in {1,...,\n} {
        \node at (\i*\cellsize - \cellsize/2, \n*\cellsize - \i*\cellsize + \cellsize/2) {\contour{white}{0}};
    }

    \draw[myrect] (0, 0) rectangle ++(\cellsize, {(\n-1)*\cellsize});
    \node at (0.5*\cellsize, {(\n-1)/2*\cellsize}) {$\vector{q}_0$};

    \draw[myrect] (\cellsize, \n*\cellsize - 4*\cellsize) rectangle ++(\cellsize, 2*\cellsize);
    \node at (1.5*\cellsize, \n*\cellsize - 3*\cellsize) {$\vector{q}_1$};

    \draw[myrect] (\cellsize, 0) rectangle ++(\cellsize, {(\n-4)*\cellsize});
    \draw[shaded] (\cellsize, 0) rectangle ++(\cellsize, {(\n-4)*\cellsize});
    \node at (1.5*\cellsize, {(\n-4)/2*\cellsize}) {\contour{white}{$\matrix{A}\vector{e}_2$}};

    \draw[myrect] (2*\cellsize, {(\n-7)*\cellsize}) rectangle ++(\cellsize, {3*\cellsize});
    \node at (2.5*\cellsize, {(\n-7+1.5)*\cellsize}) {$\vector{q}_1$};

    \draw[myrect] (2*\cellsize, 0) rectangle ++(\cellsize, {(\n-7)*\cellsize});
    \draw[shaded] (2*\cellsize, 0) rectangle ++(\cellsize, {(\n-7)*\cellsize});
    \node at (2.5*\cellsize, {(\n-7)/2*\cellsize}) {\contour{white}{$\matrix{A}\vector{e}_3$}};

    \draw[myrect] (3*\cellsize, {(\n-7)*\cellsize}) rectangle ++(\cellsize, {3*\cellsize});
    \node at (3.5*\cellsize, {(\n-7+1.5)*\cellsize}) {$\vector{q}_2$};

    \draw[myrect] (3*\cellsize, 0) rectangle ++(\cellsize, {(\n-7)*\cellsize});
    \draw[shaded] (3*\cellsize, 0) rectangle ++(\cellsize, {(\n-7)*\cellsize});
    \node at (3.5*\cellsize, {(\n-7)/2*\cellsize}) {\contour{white}{$\matrix{A}\vector{e}_4$}};

    \draw[myrect] (4*\cellsize, 0) rectangle ++(\cellsize, {(\n-7)*\cellsize});
    \node at (4.5*\cellsize, {(\n-7)/2*\cellsize}) {$\vector{q}_1$};

    \draw[myrect] (5*\cellsize, 0) rectangle ++(\cellsize, {(\n-7)*\cellsize});
    \node at (5.5*\cellsize, {(\n-7)/2*\cellsize}) {$\vector{q}_2$};

    \draw[myrect] (6*\cellsize, 0) rectangle ++(\cellsize, {(\n-7)*\cellsize});
    \node at (6.5*\cellsize, {(\n-7)/2*\cellsize}) {$\vector{q}_3$};

\end{tikzpicture}
        \caption{Spatial configuration of the $Q$-system equations on the Butcher tableau $\matrix{A}$ for $p=10$. The matrix shows the distribution of different equation types: $\vector{q}_k$ regions and $\matrix{A} \vector{e}_{j}$ regions. Blank areas in the $Q$-region represent the entries that are initialized to zero, and shaded areas indicate entries that become zero as a direct consequence of solving the constructed $Q$-system equations. Dashed lines represent the boundary of the $Q$-region.}
        \label{fig:tikz_q_conditions_Q4_p10}
    \end{figure}
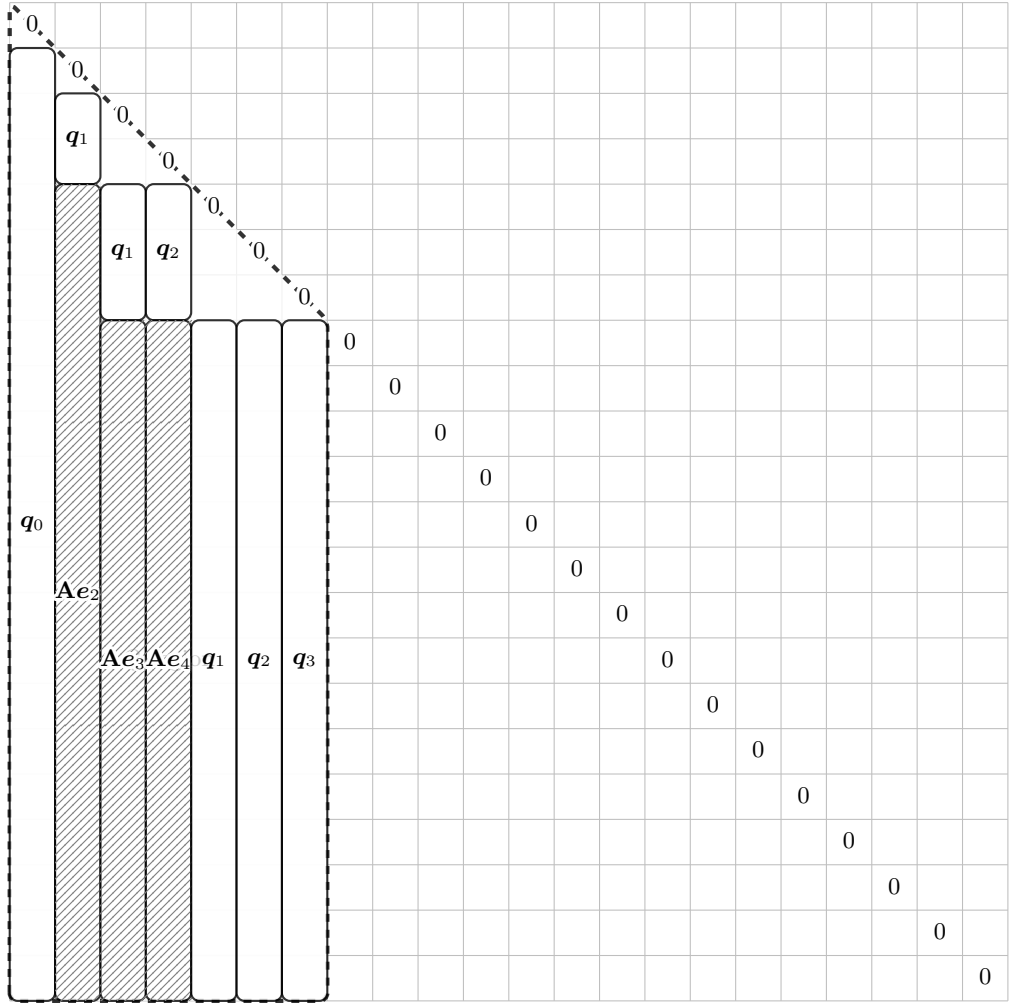

\section{Assembly of the linear systems}

\subsection{Assembly of the $D$-system}

We give an example of a linear system assembly, when $p = 10$.

\paragraph{Assembly of the $D$-system}

The variables and equations in the $D$-system are introduced in the example with $p=10$. To assemble them into a structured linear system, we adopt a particular ordering of the variables. For elements $(i,j)$ in the variable index set $I^D$, we define a partial ordering $(i,j) > (i',j')$ if and only if $j > j'$ or $j = j'$ and $i > i'$. The variables are then ordered as $(s,s-1)$, $(s,s-2)$, $(s-1,s-2)$, $(s,s-3)$, $(s-1,s-3)$, $(s,s-4)$, $(s-1,s-4)$, $(s-2,s-4)$, $(s-3,s-4)$, $\dots$. As shown in \Cref{fig:assembly_D_10}, we array them explicitly column by column. The RK method to be constructed has $s = 22$ stages. The equations are listed in \Cref{tab:assembly_D_10}.

\begin{figure}[h!]
    \centering
    \input{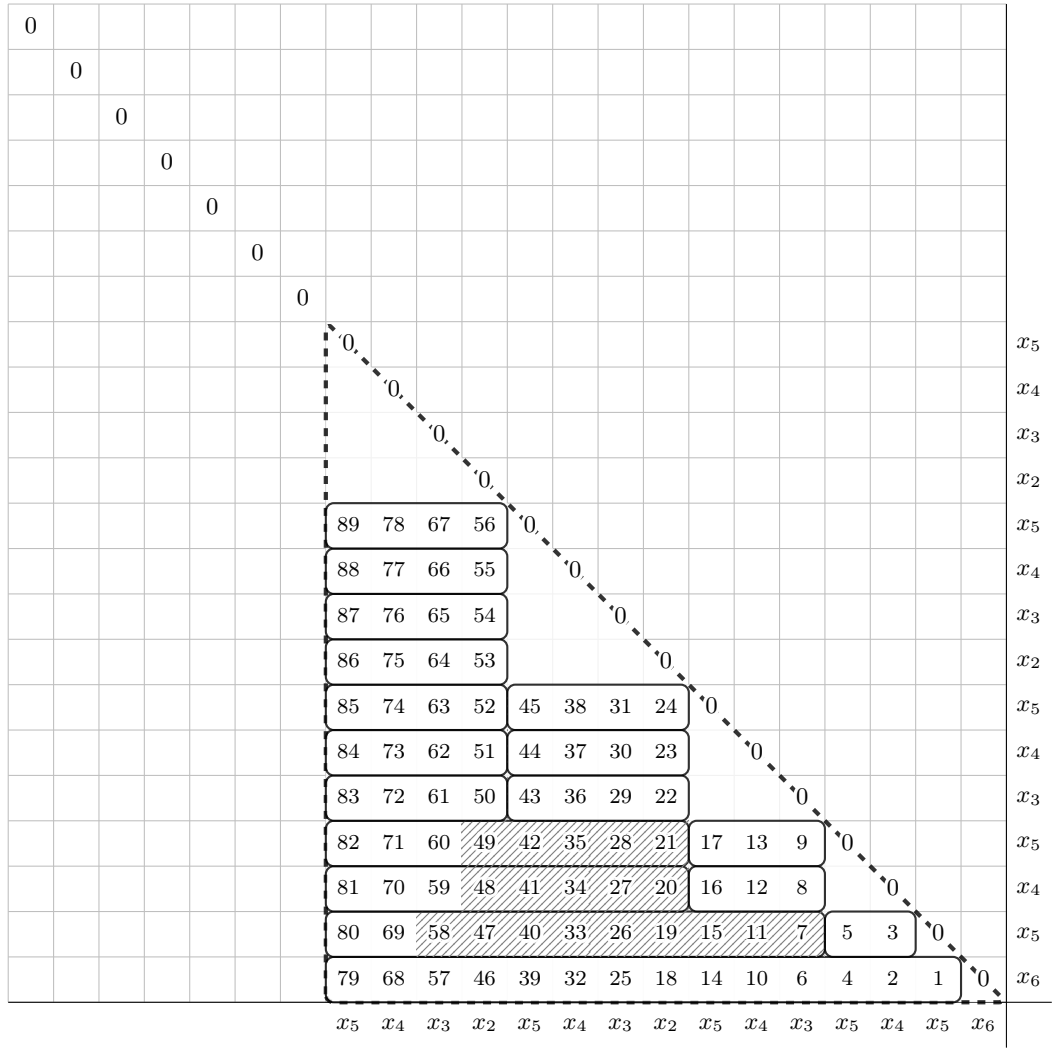}
    \caption{The ordering of the variables in the $D$-system.}
    \label{fig:assembly_D_10}
\end{figure}

\begin{table}[!h]
    \centering
    \caption{The assembly of the $D$-system for $p = 10$.}
    \begin{tabular}[t]{cc}
        \toprule
        Variables & Equations \\
        \midrule
        $a_{22, 21}$ & $\vector{d}_0 \cdot \vector{e}_{21} = 0$ \\ \midrule
        $a_{22, 20}$ & $\vector{d}_0 \cdot \vector{e}_{20} = 0$ \\ 
        $a_{21, 20}$ & $\vector{d}_1 \cdot \vector{e}_{20} = 0$ \\ \midrule
        $a_{22, 19}$ & $\vector{d}_0 \cdot \vector{e}_{19} = 0$ \\ 
        $a_{21, 19}$ & $\vector{d}_1 \cdot \vector{e}_{19} = 0$ \\ \midrule
        $a_{22, 18}$ & $\vector{d}_0 \cdot \vector{e}_{18} = 0$ \\ 
        $a_{21, 18}$ & $\vector{\varepsilon}_{1,1} \matrix{A} \cdot \vector{e}_{18} = 0$ \\ 
        $a_{20, 18}$ & $\vector{d}_1 \cdot \vector{e}_{18} = 0$ \\ 
        $a_{19, 18}$ & $\vector{d}_2 \cdot \vector{e}_{18} = 0$ \\ \midrule
        $a_{22, 17}$ & $\vector{d}_0 \cdot \vector{e}_{17} = 0$ \\ 
        $a_{21, 17}$ & $\vector{\varepsilon}_{1,1} \matrix{A} \cdot \vector{e}_{17} = 0$ \\ 
        $a_{20, 17}$ & $\vector{d}_1 \cdot \vector{e}_{17} = 0$ \\ 
        $a_{19, 17}$ & $\vector{d}_2 \cdot \vector{e}_{17} = 0$ \\ \midrule
        $a_{22, 16}$ & $\vector{d}_0 \cdot \vector{e}_{16} = 0$ \\ 
        $a_{21, 16}$ & $\vector{\varepsilon}_{1,1} \matrix{A} \cdot \vector{e}_{16} = 0$ \\ 
        $a_{20, 16}$ & $\vector{d}_1 \cdot \vector{e}_{16} = 0$ \\ 
        $a_{19, 16}$ & $\vector{d}_2 \cdot \vector{e}_{16} = 0$ \\ \midrule
        $a_{22, 15}$ & $\vector{d}_0 \cdot \vector{e}_{15} = 0$ \\
        $a_{21, 15}$ & $\vector{\varepsilon}_{1,1} \matrix{A} \cdot \vector{e}_{15} = 0$ \\ 
        $a_{20, 15}$ & $\vector{\varepsilon}_{2,1} \matrix{A} \cdot \vector{e}_{15} = 0$ \\ 
        $a_{19, 15}$ & $\vector{\varepsilon}_{2,2} \matrix{A} \cdot \vector{e}_{15} = 0$ \\ 
        $a_{18, 15}$ & $\vector{d}_1 \cdot \vector{e}_{15} = 0$ \\ 
        $a_{17, 15}$ & $\vector{d}_2 \cdot \vector{e}_{15} = 0$ \\ 
        $a_{16, 15}$ & $\vector{d}_3 \cdot \vector{e}_{15} = 0$ \\ \midrule
        $a_{22, 14}$ & $\vector{d}_0 \cdot \vector{e}_{14} = 0$ \\
        $a_{21, 14}$ & $\vector{\varepsilon}_{1,1} \matrix{A} \cdot \vector{e}_{14} = 0$ \\ 
        $a_{20, 14}$ & $\vector{\varepsilon}_{2,1} \matrix{A} \cdot \vector{e}_{14} = 0$ \\ 
        $a_{19, 14}$ & $\vector{\varepsilon}_{2,2} \matrix{A} \cdot \vector{e}_{14} = 0$ \\ 
        $a_{18, 14}$ & $\vector{d}_1 \cdot \vector{e}_{14} = 0$ \\ 
        $a_{17, 14}$ & $\vector{d}_2 \cdot \vector{e}_{14} = 0$ \\ 
        $a_{16, 14}$ & $\vector{d}_3 \cdot \vector{e}_{14} = 0$ \\             
        \bottomrule
    \end{tabular}
    \quad
    \begin{tabular}[t]{cc}
        \toprule
        Variables & Equations \\
        \midrule
        $a_{22, 13}$ & $\vector{d}_0 \cdot \vector{e}_{13} = 0$ \\
        $a_{21, 13}$ & $\vector{\varepsilon}_{1,1} \matrix{A} \cdot \vector{e}_{13} = 0$ \\ 
        $a_{20, 13}$ & $\vector{\varepsilon}_{2,1} \matrix{A} \cdot \vector{e}_{13} = 0$ \\ 
        $a_{19, 13}$ & $\vector{\varepsilon}_{2,2} \matrix{A} \cdot \vector{e}_{13} = 0$ \\ 
        $a_{18, 13}$ & $\vector{d}_1 \cdot \vector{e}_{13} = 0$ \\ 
        $a_{17, 13}$ & $\vector{d}_2 \cdot \vector{e}_{13} = 0$ \\ 
        $a_{16, 13}$ & $\vector{d}_3 \cdot \vector{e}_{13} = 0$ \\ \midrule
        $a_{22, 12}$ & $\vector{d}_0 \cdot \vector{e}_{12} = 0$ \\
        $a_{21, 12}$ & $\vector{\varepsilon}_{1,1} \matrix{A} \cdot \vector{e}_{12} = 0$ \\ 
        $a_{20, 12}$ & $\vector{\varepsilon}_{2,1} \matrix{A} \cdot \vector{e}_{12} = 0$ \\ 
        $a_{19, 12}$ & $\vector{\varepsilon}_{2,2} \matrix{A} \cdot \vector{e}_{12} = 0$ \\ 
        $a_{18, 12}$ & $\vector{d}_1 \cdot \vector{e}_{12} = 0$ \\ 
        $a_{17, 12}$ & $\vector{d}_2 \cdot \vector{e}_{12} = 0$ \\ 
        $a_{16, 12}$ & $\vector{d}_3 \cdot \vector{e}_{12} = 0$ \\ \midrule
        $a_{22, 11}$ & $\vector{d}_0 \cdot \vector{\mu}_{4} = 0$ \\
        $a_{21, 11}$ & $\vector{\varepsilon}_{1,1} \matrix{A} \cdot \vector{\mu}_{4} = 0$ \\
        $a_{20, 11}$ & $\vector{\varepsilon}_{2,1} \matrix{A} \cdot \vector{\mu}_{4} = 0$ \\
        $a_{19, 11}$ & $\vector{\varepsilon}_{2,2} \matrix{A} \cdot \vector{\mu}_{4} = 0$ \\
        $a_{18, 11}$ & $\vector{\varepsilon}_{3,1} \matrix{A} \cdot \vector{\mu}_{4} = 0$ \\
        $a_{17, 11}$ & $\vector{\varepsilon}_{3,2} \matrix{A} \cdot \vector{\mu}_{4} = 0$ \\
        $a_{16, 11}$ & $\vector{\varepsilon}_{3,3} \matrix{A} \cdot \vector{\mu}_{4} = 0$ \\
        $a_{15, 11}$ & $\vector{d}_1 \cdot \vector{\mu}_{4} = 0$ \\
        $a_{14, 11}$ & $\vector{d}_2 \cdot \vector{\mu}_{4} = 0$ \\
        $a_{13, 11}$ & $\vector{d}_3 \cdot \vector{\mu}_{4} = 0$ \\
        $a_{12, 11}$ & $\vector{d}_4 \cdot \vector{\mu}_{4} = 0$ \\ 
        \bottomrule
    \end{tabular}
    \quad
    \begin{tabular}[t]{cc}
        \toprule
        Variables & Equations \\
        \midrule
        $a_{22, 10}$ & $\vector{d}_0 \cdot \vector{\mu}_{3} = 0$ \\
        $a_{21, 10}$ & $\vector{\varepsilon}_{1,1} \matrix{A} \cdot \vector{\mu}_{3} = 0$ \\
        $a_{20, 10}$ & $\vector{\varepsilon}_{2,1} \matrix{A} \cdot \vector{\mu}_{3} = 0$ \\
        $a_{19, 10}$ & $\vector{\varepsilon}_{2,2} \matrix{A} \cdot \vector{\mu}_{3} = 0$ \\
        $a_{18, 10}$ & $\vector{\varepsilon}_{3,1} \matrix{A} \cdot \vector{\mu}_{3} = 0$ \\
        $a_{17, 10}$ & $\vector{\varepsilon}_{3,2} \matrix{A} \cdot \vector{\mu}_{3} = 0$ \\
        $a_{16, 10}$ & $\vector{\varepsilon}_{3,3} \matrix{A} \cdot \vector{\mu}_{3} = 0$ \\
        $a_{15, 10}$ & $\vector{d}_1 \cdot \vector{\mu}_{3} = 0$ \\
        $a_{14, 10}$ & $\vector{d}_2 \cdot \vector{\mu}_{3} = 0$ \\
        $a_{13, 10}$ & $\vector{d}_3 \cdot \vector{\mu}_{3} = 0$ \\
        $a_{12, 10}$ & $\vector{d}_4 \cdot \vector{\mu}_{3} = 0$ \\ \midrule
        $a_{22, 9}$ & $\vector{d}_0 \cdot \vector{\mu}_{2} = 0$ \\
        $a_{21, 9}$ & $\vector{\varepsilon}_{1,1} \matrix{A} \cdot \vector{\mu}_{2} = 0$ \\
        $a_{20, 9}$ & $\vector{\varepsilon}_{2,1} \matrix{A} \cdot \vector{\mu}_{2} = 0$ \\
        $a_{19, 9}$ & $\vector{\varepsilon}_{2,2} \matrix{A} \cdot \vector{\mu}_{2} = 0$ \\
        $a_{18, 9}$ & $\vector{\varepsilon}_{3,1} \matrix{A} \cdot \vector{\mu}_{2} = 0$ \\
        $a_{17, 9}$ & $\vector{\varepsilon}_{3,2} \matrix{A} \cdot \vector{\mu}_{2} = 0$ \\
        $a_{16, 9}$ & $\vector{\varepsilon}_{3,3} \matrix{A} \cdot \vector{\mu}_{2} = 0$ \\
        $a_{15, 9}$ & $\vector{d}_1 \cdot \vector{\mu}_{2} = 0$ \\
        $a_{14, 9}$ & $\vector{d}_2 \cdot \vector{\mu}_{2} = 0$ \\
        $a_{13, 9}$ & $\vector{d}_3 \cdot \vector{\mu}_{2} = 0$ \\
        $a_{12, 9}$ & $\vector{d}_4 \cdot \vector{\mu}_{2} = 0$ \\ \midrule
        $a_{22, 8}$ & $\vector{d}_0 \cdot \vector{\mu}_{1} = 0$ \\
        $a_{21, 8}$ & $\vector{\varepsilon}_{1,1} \matrix{A} \cdot \vector{\mu}_{1} = 0$ \\
        $a_{20, 8}$ & $\vector{\varepsilon}_{2,1} \matrix{A} \cdot \vector{\mu}_{1} = 0$ \\
        $a_{19, 8}$ & $\vector{\varepsilon}_{2,2} \matrix{A} \cdot \vector{\mu}_{1} = 0$ \\
        $a_{18, 8}$ & $\vector{\varepsilon}_{3,1} \matrix{A} \cdot \vector{\mu}_{1} = 0$ \\
        $a_{17, 8}$ & $\vector{\varepsilon}_{3,2} \matrix{A} \cdot \vector{\mu}_{1} = 0$ \\
        $a_{16, 8}$ & $\vector{\varepsilon}_{3,3} \matrix{A} \cdot \vector{\mu}_{1} = 0$ \\
        $a_{15, 8}$ & $\vector{d}_1 \cdot \vector{\mu}_{1} = 0$ \\
        $a_{14, 8}$ & $\vector{d}_2 \cdot \vector{\mu}_{1} = 0$ \\
        $a_{13, 8}$ & $\vector{d}_3 \cdot \vector{\mu}_{1} = 0$ \\
        $a_{12, 8}$ & $\vector{d}_4 \cdot \vector{\mu}_{1} = 0$ \\ \bottomrule
    \end{tabular}
    \label{tab:assembly_D_10}
\end{table}

Using the definition $\vector{d}_n = (\vector{b} \cdot \vector{c}^{\odot n}) \times^1 \matrix{A} - \frac{1}{n+1} \vector{b} \odot (\vector{1} - \vector{c}^{\odot (n+1)})$ and the definition of basis vectors $\vector{e}_i$ and $\vector{\varepsilon}_{i,j}$, the system can be assembled as follows:

\begin{equation}
    \matrix{M}_D = 
    \left[
    \begin{array}{*{14}{>{\scriptstyle}c@{\hspace{1mm}}}}
        \matrix{L}_{1,1}^D & 0 & 0 & 0 & 0 & 0 & 0 & 0 & 0 & 0 & 0 & 0 & 0 & 0 \\
        0 & \matrix{L}_{2,2}^D & 0 & 0 & 0 & 0 & 0 & 0 & 0 & 0 & 0 & 0 & 0 & 0 \\
        0 & 0 & \matrix{L}_{3,3}^D & 0 & 0 & 0 & 0 & 0 & 0 & 0 & 0 & 0 & 0 & 0 \\
        0 & 0 & 0 & \matrix{L}_{4,4}^D & 0 & 0 & 0 & 0 & 0 & 0 & 0 & 0 & 0 & 0 \\
        0 & 0 & 0 & 0 & \matrix{L}_{5,5}^D & 0 & 0 & 0 & 0 & 0 & 0 & 0 & 0 & 0 \\
        0 & 0 & 0 & 0 & 0 & \matrix{L}_{6,6}^D & 0 & 0 & 0 & 0 & 0 & 0 & 0 & 0 \\
        0 & 0 & 0 & 0 & 0 & 0 & \matrix{L}_{7,7}^D & 0 & 0 & 0 & 0 & 0 & 0 & 0 \\
        0 & 0 & 0 & 0 & 0 & 0 & 0 & \matrix{L}_{8,8}^D & 0 & 0 & 0 & 0 & 0 & 0 \\
        0 & 0 & 0 & 0 & 0 & 0 & 0 & 0 & \matrix{L}_{9,9}^D & 0 & 0 & 0 & 0 & 0 \\
        0 & 0 & 0 & 0 & 0 & 0 & 0 & 0 & 0 & \matrix{L}_{10,10}^D & 0 & 0 & 0 & 0 \\
        0 & 0 & 0 & 0 & 0 & 0 & \matrix{L}_{11,7}^D & 0 & 0 & 0 & \matrix{L}_{11,11}^D & 0 & 0 & 0 \\
        0 & 0 & 0 & \matrix{L}_{12,4}^D & 0 & 0 & 0 & \matrix{L}_{12,8}^D & 0 & 0 & 0 & \matrix{L}_{12,12}^D & 0 & 0 \\
        0 & \matrix{L}_{13,2}^D & 0 & 0 & \matrix{L}_{13,5}^D & 0 & 0 & 0 & \matrix{L}_{13,9}^D & 0 & 0 & 0 & \matrix{L}_{13,13}^D & 0 \\
        \matrix{L}_{14,1}^D & 0 & \matrix{L}_{14,3}^D & 0 & 0 & \matrix{L}_{14,6}^D & 0 & 0 & 0 & \matrix{L}_{14,10}^D & 0 & 0 & 0 & \matrix{L}_{14,14}^D 
    \end{array}
    \right]
\end{equation}
where 
\begin{equation}
    \matrix{L}_{1,1}^D = 
    \left[
    \begin{array}{*{1}{>{\textstyle}c}}
        b_{22} \\
    \end{array}
    \right], \quad
    \matrix{L}_{2,2}^D = \matrix{L}_{3,3}^D =
    \left[
    \begin{array}{*{2}{>{\textstyle}c}}
        b_{22} & b_{21} \\
        b_{22} c_{22} & b_{21} c_{21} \\
    \end{array}
    \right],
\end{equation}
\begin{equation}
    \matrix{L}_{4,4}^D = \matrix{L}_{5,5}^D = \matrix{L}_{6,6}^D = \left[
    \begin{array}{*{4}{>{\textstyle}c}}
        b_{22} & b_{21} & b_{20} & b_{19} \\
        0 & 1 & 0 & 0 \\
        b_{22} c_{22} & b_{21} c_{21} & b_{20} c_{20} & b_{19} c_{19} \\
        b_{22} c_{22}^2 & b_{21} c_{21}^2 & b_{20} c_{20}^2 & b_{19} c_{19}^2
    \end{array}
    \right],
\end{equation}
\begin{equation}
    \matrix{L}_{7,7}^D = \matrix{L}_{8,8}^D = \matrix{L}_{9,9}^D = \matrix{L}_{10,10}^D = \left[
    \begin{array}{*{7}{>{\textstyle}c}}
        b_{22} & b_{21} & b_{20} & b_{19} & b_{18} & b_{17} & b_{16} \\
        0 & 1 & 0 & 0 & 0 & 0 & 0 \\
        0 & 0 & 1 & 0 & 0 & 0 & 0 \\
        0 & 0 & 0 & 1 & 0 & 0 & 0 \\
        b_{22} c_{22} & b_{21} c_{21} & b_{20} c_{20} & b_{19} c_{19} & b_{18} c_{18} & b_{17} c_{17} & b_{16} c_{16} \\
        b_{22} c_{22}^2 & b_{21} c_{21}^2 & b_{20} c_{20}^2 & b_{19} c_{19}^2 & b_{18} c_{18}^2 & b_{17} c_{17}^2 & b_{16} c_{16}^2 \\
        b_{22} c_{22}^3 & b_{21} c_{21}^3 & b_{20} c_{20}^3 & b_{19} c_{19}^3 & b_{18} c_{18}^3 & b_{17} c_{17}^3 & b_{16} c_{16}^3
    \end{array}
    \right],
\end{equation}
\begin{equation}
    \begin{split}
        & \matrix{L}_{11,11}^D = \matrix{L}_{12,12}^D = \matrix{L}_{13,13}^D = \matrix{L}_{14,14}^D \\
        &= \left[
        \begin{array}{*{11}{>{\textstyle}c@{\hspace{2mm}}}}
            b_{22} & b_{21} & b_{20} & b_{19} & b_{18} & b_{17} & b_{16} & b_{15} & b_{14} & b_{13} & b_{12} \\
            0 & 1 & 0 & 0 & 0 & 0 & 0 & 0 & 0 & 0 & 0 \\
            0 & 0 & 1 & 0 & 0 & 0 & 0 & 0 & 0 & 0 & 0 \\
            0 & 0 & 0 & 1 & 0 & 0 & 0 & 0 & 0 & 0 & 0 \\
            0 & 0 & 0 & 0 & 1 & 0 & 0 & 0 & 0 & 0 & 0 \\
            0 & 0 & 0 & 0 & 0 & 1 & 0 & 0 & 0 & 0 & 0 \\
            0 & 0 & 0 & 0 & 0 & 0 & 1 & 0 & 0 & 0 & 0 \\
            b_{22} c_{22} & b_{21} c_{21} & b_{20} c_{20} & b_{19} c_{19} & b_{18} c_{18} & b_{17} c_{17} & b_{16} c_{16} & b_{15} c_{15} & b_{14} c_{14} & b_{13} c_{13} & b_{12} c_{12} \\
            b_{22} c_{22}^2 & b_{21} c_{21}^2 & b_{20} c_{20}^2 & b_{19} c_{19}^2 & b_{18} c_{18}^2 & b_{17} c_{17}^2 & b_{16} c_{16}^2 & b_{15} c_{15}^2 & b_{14} c_{14}^2 & b_{13} c_{13}^2 & b_{12} c_{12}^2 \\
            b_{22} c_{22}^3 & b_{21} c_{21}^3 & b_{20} c_{20}^3 & b_{19} c_{19}^3 & b_{18} c_{18}^3 & b_{17} c_{17}^3 & b_{16} c_{16}^3 & b_{15} c_{15}^3 & b_{14} c_{14}^3 & b_{13} c_{13}^3 & b_{12} c_{12}^3
        \end{array}
        \right],
    \end{split}
\end{equation}
\begin{equation}
    \matrix{L}_{11,7}^D = \matrix{L}_{12,8}^D = \matrix{L}_{13,9}^D = \matrix{L}_{14,10}^D
    = \left[
    \begin{array}{*{7}{>{\textstyle}c }}
        b_{22} & b_{21} & b_{20} & b_{19} & b_{18} & b_{17} & b_{16} \\
        0 & 1 & 0 & 0 & 0 & 0 & 0 \\
        0 & 0 & 1 & 0 & 0 & 0 & 0 \\
        0 & 0 & 0 & 1 & 0 & 0 & 0 \\
        0 & 0 & 0 & 0 & 1 & 0 & 0 \\
        0 & 0 & 0 & 0 & 0 & 1 & 0 \\
        0 & 0 & 0 & 0 & 0 & 0 & 1 \\
        b_{22} c_{22} & b_{21} c_{21} & b_{20} c_{20} & b_{19} c_{19} & b_{18} c_{18} & b_{17} c_{17} & b_{16} c_{16} \\
        b_{22} c_{22}^2 & b_{21} c_{21}^2 & b_{20} c_{20}^2 & b_{19} c_{19}^2 & b_{18} c_{18}^2 & b_{17} c_{17}^2 & b_{16} c_{16}^2 \\
        b_{22} c_{22}^3 & b_{21} c_{21}^3 & b_{20} c_{20}^3 & b_{19} c_{19}^3 & b_{18} c_{18}^3 & b_{17} c_{17}^3 & b_{16} c_{16}^3 \\
        b_{22} c_{22}^4 & b_{21} c_{21}^4 & b_{20} c_{20}^4 & b_{19} c_{19}^4 & b_{18} c_{18}^4 & b_{17} c_{17}^4 & b_{16} c_{16}^4 \\
    \end{array}
    \right],
\end{equation}
\begin{equation}
    \matrix{L}_{12,4}^D = \matrix{L}_{13,5}^D = \matrix{L}_{14,6}^D
    = \left[
    \begin{array}{*{4}{>{\textstyle}c }}
        b_{22} & b_{21} & b_{20} & b_{19} \\
        0 & 1 & 0 & 0 \\
        0 & 0 & 1 & 0 \\
        0 & 0 & 0 & 1 \\
        0 & 0 & 0 & 0 \\
        0 & 0 & 0 & 0 \\
        0 & 0 & 0 & 0 \\
        b_{22} c_{22} & b_{21} c_{21} & b_{20} c_{20} & b_{19} c_{19} \\
        b_{22} c_{22}^2 & b_{21} c_{21}^2 & b_{20} c_{20}^2 & b_{19} c_{19}^2 \\
        b_{22} c_{22}^3 & b_{21} c_{21}^3 & b_{20} c_{20}^3 & b_{19} c_{19}^3 \\
        b_{22} c_{22}^4 & b_{21} c_{21}^4 & b_{20} c_{20}^4 & b_{19} c_{19}^4 \\
    \end{array}
    \right],
\end{equation}
\begin{equation}
    \matrix{L}_{13,2}^D = \matrix{L}_{14,3}^D
    = \left[
    \begin{array}{*{2}{>{\textstyle}c }}
        b_{22} & b_{21}\\
        0 & 1 \\
        0 & 0 \\
        0 & 0 \\
        0 & 0 \\
        0 & 0 \\
        0 & 0 \\
        b_{22} c_{22} & b_{21} c_{21} \\
        b_{22} c_{22}^2 & b_{21} c_{21}^2 \\
        b_{22} c_{22}^3 & b_{21} c_{21}^3 \\
        b_{22} c_{22}^4 & b_{21} c_{21}^4 \\
    \end{array}
    \right],
\end{equation}
\begin{equation}
    \matrix{L}_{14,1}^D
    = \left[
    \begin{array}{*{1}{>{\textstyle}c }}
        b_{22}\\
        0 \\
        0 \\
        0 \\
        0 \\
        0 \\
        0 \\
        b_{22} c_{22} \\
        b_{22} c_{22}^2 \\
        b_{22} c_{22}^3 \\
        b_{22} c_{22}^4 \\
    \end{array}
    \right].
\end{equation}

The right-hand side $\vector{r}_D$ of the $D$-system is also assembled from the equations in \Cref{tab:assembly_D_10}. We omit the details of the exact formulas for $\vector{r}_D$ since they have little affect on the understanding of the theory and algorithms we present.

\begin{remark*}
    In general, the linear system for the $D$-system is a block lower-triangular matrix. Upon certain reordering of the variables, any diagonal block can be written as 
    \begin{equation}
        \matrix{L}_{i,i}^D = \begin{bmatrix}
            \matrix{V}_{i,i}^D & \matrix{W}_{i,i}^D \\
            \matrix{0} & \matrix{I}
        \end{bmatrix}
    \end{equation}
    where $\matrix{V}_{i,i}^D$ is a square Vandermonde matrix and $\matrix{W}_{i,i}^D$ is a non-square Vandermonde matrix.
\end{remark*}

\subsection{Assembly of the $Q$-system}

The variables and equations for the $Q$-system are introduced in the example with $p=10$. To assemble them into a structured linear system, we adopt a particular ordering of the variables. For elements $(i,j)$ in the variable index set $I^Q$,  we define a partial ordering `$<$': $(i,j) < (i',j')$ if and only if $i < i'$ or $i = i'$ and $j < j'$. The variables are then ordered as $(2,1)$, $(3,1)$, $(3,2)$, $(4,1)$, $(4,2)$, $(5,1)$, $(5,2)$, $(5,3)$, $(5,4)$, $\dots$. As shown in \Cref{fig:assembly_Q_10}, we array them explicitly row by row. The RK method for $p = 10$ has $s = 22$ stages. The equations for each variable is given in \Cref{tab:assembly_Q_10}.

\begin{figure}[h!]
    \centering
    \input{tikz/tikz_q_variable_ordering_Q4_p10_QDWithClustersMinimalV1.tex}
    \caption{The ordering of the variables in the $Q$-system.}
    \label{fig:assembly_Q_10}
\end{figure}

\begin{table}[!h]
    \centering
    \caption{The equations for each variable in the $Q$-system.}
    \begin{tabular}[t]{cc}
        \toprule
        Variables & Equations \\
        \midrule
        $a_{2,1}$ & $\vector{q}_0 \cdot \vector{e}_2 = 0$ \\ \midrule
        $a_{3,1}$ & $\vector{q}_0 \cdot \vector{e}_3 = 0$ \\ 
        $a_{3,2}$ & $\vector{q}_1 \cdot \vector{e}_3 = 0$ \\ \midrule
        $a_{4,1}$ & $\vector{q}_0 \cdot \vector{e}_4 = 0$ \\ 
        $a_{4,2}$ & $\vector{q}_1 \cdot \vector{e}_4 = 0$ \\ \midrule
        $a_{5,1}$ & $\vector{q}_0 \cdot \vector{e}_5 = 0$ \\ 
        $a_{5,2}$ & $(\matrix{A} \vector{e}_2) \cdot \vector{e}_5 = 0$ \\ 
        $a_{5,3}$ & $\vector{q}_1 \cdot \vector{e}_5 = 0$ \\ 
        $a_{5,4}$ & $\vector{q}_2 \cdot \vector{e}_5 = 0$ \\ \midrule
        $a_{6,1}$ & $\vector{q}_0 \cdot \vector{e}_6 = 0$ \\ 
        $a_{6,2}$ & $(\matrix{A} \vector{e}_2) \cdot \vector{e}_6 = 0$ \\ 
        $a_{6,3}$ & $\vector{q}_1 \cdot \vector{e}_6 = 0$ \\ 
        $a_{6,4}$ & $\vector{q}_2 \cdot \vector{e}_6 = 0$ \\ \midrule
        $a_{7,1}$ & $\vector{q}_0 \cdot \vector{e}_7 = 0$ \\ 
        $a_{7,2}$ & $(\matrix{A} \vector{e}_2) \cdot \vector{e}_7 = 0$ \\ 
        $a_{7,3}$ & $\vector{q}_1 \cdot \vector{e}_7 = 0$ \\ 
        $a_{7,4}$ & $\vector{q}_2 \cdot \vector{e}_7 = 0$ \\ \midrule
        \multicolumn{2}{c}{For $8 \leq i \leq 22$:} \\
        $a_{i,1}$ & $\vector{q}_0 \cdot \vector{e}_i = 0$ \\ 
        $a_{i,2}$ & $(\matrix{A} \vector{e}_2) \cdot \vector{e}_i = 0$ \\ 
        $a_{i,3}$ & $(\matrix{A} \vector{e}_3) \cdot \vector{e}_i = 0$ \\ 
        $a_{i,4}$ & $(\matrix{A} \vector{e}_4) \cdot \vector{e}_i = 0$ \\ 
        $a_{i,5}$ & $\vector{q}_1 \cdot \vector{e}_i = 0$ \\ 
        $a_{i,6}$ & $\vector{q}_2 \cdot \vector{e}_i = 0$ \\ 
        $a_{i,7}$ & $\vector{q}_3 \cdot \vector{e}_i = 0$ \\ 
        \bottomrule
    \end{tabular}
    \label{tab:assembly_Q_10}
\end{table}

Using the definition of $\vector{q}_n = \matrix{A} \vector{c}^{\odot n} - \frac{1}{n+1} \vector{c}^{\odot n+1}$. The system can be assembled as follows:

\begin{equation}
    \matrix{M}_Q = \operatorname{BlockDiag}\left( \matrix{L}_2^Q, \matrix{L}_3^Q,\, \dots, \, \matrix{L}_{22}^Q \right)
\end{equation}
where the subscripts indicate which row the variable belongs to, and 
\begin{equation}
    \matrix{L}_2^Q = \begin{bmatrix}
        1
    \end{bmatrix}, \quad 
    \matrix{L}_3^Q = \matrix{L}_4^Q = \begin{bmatrix}
        1 & 1 \\
        c_1 & c_2
    \end{bmatrix}, 
\end{equation}
\begin{equation}
    \matrix{L}_5^Q = \matrix{L}_6^Q = \matrix{L}_7^Q = \begin{bmatrix}
        1 & 1 & 1 & 1 \\
        0 & 1 & 0 & 0 \\
        c_1 & c_2 & c_3 & c_4 \\
        c_1^2 & c_2^2 & c_3^2 & c_4^2 \\
    \end{bmatrix}, 
\end{equation}
\begin{equation}
    \matrix{L}_i^Q = \begin{bmatrix}
        1 & 1 & 1 & 1 & 1 & 1 & 1 & 1 \\
        0 & 1 & 0 & 0 & 0 & 0 & 0 & 0 \\
        0 & 0 & 1 & 0 & 0 & 0 & 0 & 0 \\
        0 & 0 & 0 & 1 & 0 & 0 & 0 & 0 \\
        c_1 & c_2 & c_3 & c_4 & c_5 & c_6 & c_7 & c_8 \\
        c_1^2 & c_2^2 & c_3^2 & c_4^2 & c_5^2 & c_6^2 & c_7^2 & c_8^2 \\
        c_1^3 & c_2^3 & c_3^3 & c_4^3 & c_5^3 & c_6^3 & c_7^3 & c_8^3 \\
    \end{bmatrix}, \quad \text{for } 8 \leq i \leq 22.
\end{equation}

The right hand side $\vector{r}_Q$ of the linear system is also assembled from the equations in \Cref{tab:assembly_Q_10}. It is not difficult yet tedious to write it out, so we omit the details here. For conceptual understanding of the solvability and condition number of the linear system, focusing on $\matrix{M}_Q$ is adequate. For implementation, some variables (the shadowed ones in \Cref{fig:assembly_Q_10}) can be decoupled from the system and their values can be determined a priori. Therefore, we only need to solve a smaller block diagonal linear system. Efficient algorithms can be adopted to obtain the solution.

\begin{remark*}
    In general, the linear system for the $Q$-system is a block diagonal matrix. Upon certain permutations of rows and columns, any diagonal block can be written as 
    \begin{equation}
        \matrix{L}_i^Q = \begin{bmatrix}
            \matrix{V}_i^Q & \matrix{W}_i^Q \\
            \matrix{0} & \matrix{I} 
        \end{bmatrix}
    \end{equation}
    where $\matrix{V}_i^Q$ is a square Vandermonde matrix and $\matrix{W}_i^Q$ is a non-square Vandermonde matrix.
\end{remark*}

\section{Efficiency of Construction}

The recursive formulation leads not only to ERK schemes satisfying order conditions but also to a highly efficient algorithm for generating ERK coefficients. This subsection present a mathematical analysis of its computational complexity.

\paragraph{Algorithmic structure.}
After the recursion for constructing $Q$ and $D$-type spaces is specified, all unknown coefficients of the RK tableau are determined by two linear systems:
\begin{enumerate}[(i)]
  \item a \emph{$D$-system}, derived from the construction of $D$-type spaces, which is a block lower-triangular linear system;
  \item a \emph{$Q$-system}, derived from the construction of $Q$-type spaces, which is a block diagonal linear system.
\end{enumerate}
The two systems are solved sequentially: the $D$-system is solved first, yielding all bottom-right coefficients in $\matrix{A}$, and its result is then substituted into the $Q$-system to determine left part of $\matrix{A}$. No nonlinear coupling between the two systems remains.

\paragraph{Complexity model.}
Let $p$ denote the order of an RK scheme and $s=s(p)$ the resulting number of stages ($s(p)\sim p^2/4$ asymptotically). Let the $D$-system have diagonal blocks of sizes $d_1,\ldots,d_{s_2-1}$, where $s_2 = n(n+1)/2$; the $Q$-system have diagonal blocks of sizes $q_1,\ldots,q_{s-1}$. 
The total number of unknowns is
\begin{equation*}
    \mathcal{N} = \mathcal{N}_D + \mathcal{N}_Q = \sum_{i=1}^{s_2-1} d_i + \sum_{j=1}^{s-1} q_j \le \sum_{i=1}^{s_2-1} i + \sum_{j=1}^{s-1} (l+1) = \mathcal{O}(s^2), 
\end{equation*}
where $\mathcal{N}_D$ and $\mathcal{N}_Q$ are the number of unknowns in the $D$- and $Q$-system, respectively. In the following complexity analysis, we assume that solving a dense $r\times r$ block requires $\Theta(r^3)$ arithmetic operations.

\begin{proposition}[Cost of solving the $D$-system]
\label{prop:Dsystem-cost}
If the $D$-system matrix is block lower-triangular with diagonal blocks of sizes $d_i$, then the total arithmetic cost of solving the $D$-system is
\begin{equation*}
    \mathcal{C}_D = \Theta \Big( \sum_{i=1}^{s_2-1} d_i^3 + \sum_{i=1}^{s_2-1} \sum_{j < i} d_i d_j \Big)
    \leq \mathcal{O}(s^4).
\end{equation*}
\end{proposition}

\begin{proof}
Backward substitution proceeds blockwise: each diagonal block solve costs $\Theta(d_i^3)$, and each coupling to preceding blocks adds
$\Theta(d_i d_j)$ work bounded by the same order. Summing over all blocks yields 
\begin{align*}
    \mathcal{C}_D &= \Theta\!\left( \sum_{i=1}^{s_2-1} d_i^3 + \sum_{i=1}^{s_2-1} \sum_{j < i} d_i d_j \right) 
    = \Theta\!\left( \sum_{i=1}^{s_2-1} d_i^3 + \frac{1}{2} \Big( \big( \sum_{i=1}^{s_2-1} d_i \big)^2 - \sum_{i=1}^{s_2-1} d_i^2 \Big) \right) \\
    &= \Theta\!\left( \sum_{i=1}^{s_2-1} d_i^3 + \frac{1}{2} \Big( \mathcal{N}_D^2 - \sum_{i=1}^{s_2-1} d_i^2 \Big) \right).
\end{align*}
By our construction, $d_i \leq i$, so 
\begin{align*}
    \mathcal{C}_D \leq \mathcal{O}\!\left( \sum_{i=1}^{s_2-1} i^3 + \frac{1}{2} \mathcal{N}_D^2 \right) = \mathcal{O}\!\left( s^4  \right).
\end{align*}
\end{proof}

\begin{proposition}[Cost of solving the $Q$-system]
\label{prop:Qsystem-cost}
If the $Q$-system is block diagonal with blocks of size $q_j$,
then the total cost of solving it is
\begin{equation*}
    \mathcal{C}_Q = \Theta \Big( \sum_{j=1}^{s-1} q_j^3 \Big)
    \leq \mathcal{O}( s^4 ).
\end{equation*}
\end{proposition}

\begin{proof}
Each block is independent and can be solved separately.
For block of size $q_j$, the block solve cost is $\Theta(q_j^3)$,
and summing over $s-1$ blocks yields total cost of
\begin{equation*}
    \mathcal{C}_Q = \Theta \Big( \sum_{j=1}^{s-1} q_j^3 \Big)
\end{equation*}
using the inequality $q_j \leq l+1$, we have
\begin{equation*}
    \mathcal{C}_Q \leq \mathcal{O}\!\left( \sum_{j=1}^{s-1} (l+1)^3 \right) = \mathcal{O}(s^4).
\end{equation*}
\end{proof}

\begin{theorem}[Overall computational complexity]
\label{thm:overall-complexity}
Under the structural assumptions above, the total arithmetic cost of constructing all coefficients is
\begin{equation}
\mathcal{C}_{\mathrm{total}}
= \mathcal{C}_D + \mathcal{C}_Q + \mathcal{O}(\mathcal{N})
\leq \mathcal{O}(s^4)
\end{equation}
\end{theorem}

\begin{proof}
The proof is direct from Propositions~\ref{prop:Dsystem-cost} and \ref{prop:Qsystem-cost}. The $\mathcal{O}(\mathcal{N})$ term accounts for assembly and simple vector operations.
\end{proof}

\begin{remark}
In the actual implementation of the algorithm, we further reduce the size of the linear systems by assigning the variables corresponding to the shaded regions in Figures~\ref{fig:tikz_d_conditions_D5_p10} and~\ref{fig:tikz_q_conditions_Q4_p10} in advance. This further reduces the complexity of the algorithm. In practice, the computational cost of solving both the $Q$-system and the $D$-system can be controlled within $\mathcal{O}(s^3 p)$, so the overall complexity can also be reduced to $\mathcal{O}(s^3 p)$.
\end{remark}

\section{Tables of coefficients for constructed methods}

\paragraph{Coefficients for our method of order 4.}

\begin{equation*}
    \renewcommand\arraystretch{1.4}
    \begin{array}{c|cccc}
        0 & 0 & 0 & 0 & 0 \\
        1/2 & 1/2 & 0 & 0 & 0 \\
        1/2 & 0 & 1/2 & 0 & 0 \\
        1 & 0 & 0 & 1 & 0 \\
        \hline
          & 1/6 & 1/3 & 1/3 & 1/6
    \end{array}
\end{equation*}

\paragraph{Coefficients for our method of order 6.}

\begin{equation*}
    \renewcommand\arraystretch{1.4}
    \begin{array}{c|cccccccc}
        0 & 0 & 0 & 0 & 0 & 0 & 0 & 0 & 0 \\
        c_2 & a_{21} & 0 & 0 & 0 & 0 & 0 & 0 & 0 \\
        c_3 & a_{31} & a_{32} & 0 & 0 & 0 & 0 & 0 & 0 \\
        c_4 & a_{41} & a_{42} & 0 & 0 & 0 & 0 & 0 & 0 \\
        c_5 & a_{51} & a_{52} & 1/2 & a_{54} & 0 & 0 & 0 & 0 \\
        c_6 & a_{61} & a_{62} & a_{63} & 1/3 & 0 & 0 & 0 & 0 \\
        c_7 & a_{71} & a_{72} & a_{73} & a_{74} & a_{75} & a_{76} & 0 & 0 \\
        1 & 1/6 & 0 & a_{83} & a_{84} & a_{85} & a_{86} & a_{87} & 0 \\
        \hline
        & 1/12 & 0 & 5/36 & 5/24 & 5/36 & 5/24 & 5/36 & 1/12
    \end{array}
\end{equation*}

The precise parameters containing extended floating-point representations are defined as follows:
\begingroup
\small
\begin{longtable}{rl}
    \toprule
    Parameter & Precise Numerical Value \\
    \midrule
    \endhead
    \bottomrule
    \endfoot
        $c_2 = c_4 = c_6$ & \texttt{0.2763932022500210303590826331268723764559} \\
        $c_3 = c_5 = c_7$ & \texttt{0.7236067977499789696409173668731276235441} \\
        $a_{21}$ & \texttt{0.2763932022500210303590826331268723764559} \\
        $a_{31}$ & \texttt{-0.2236067977499789696409173668731276235441} \\
        $a_{32}$ & \texttt{0.9472135954999579392818347337462552470881} \\
        $a_{41} = a_{42}$ & \texttt{0.1381966011250105151795413165634361882280} \\
        $a_{51}$ & \texttt{0.5854101966249684544613760503096914353161} \\
        $a_{52}$ & \texttt{-1.8944271909999158785636694674925104941762} \\
        $a_{54}$ & \texttt{1.5326237921249263937432107840559466824042} \\
        $a_{61}$ & \texttt{0.1030056647916491413674311390609396862867} \\
        $a_{62}$ & \texttt{-0.1381966011250105151795413165634361882280} \\
        $a_{63}$ & \texttt{-0.0217491947499509291621405227039644549361} \\
        $a_{71}$ & \texttt{-0.2236067977499789696409173668731276235441} \\
        $a_{72}$ & \texttt{0.9472135954999579392818347337462552470881} \\
        $a_{73}$ & \texttt{-0.1381966011250105151795413165634361882280} \\
        $a_{74}$ & \texttt{-1.4208203932499369089227521006193828706322} \\
        $a_{75}$ & \texttt{0.1381966011250105151795413165634361882280} \\
        $a_{76}$ & \texttt{1.4208203932499369089227521006193828706322} \\
        $a_{83}$ & \texttt{-0.0879773408334034345302754437562412548531} \\
        $a_{84}$ & \texttt{0.7893446629166316160681956114552127059068} \\
        $a_{85}$ & \texttt{0.2303276685416841919659021942723936470466} \\
        $a_{86}$ & \texttt{-0.5590169943749474241022934171828190588602} \\
        $a_{87}$ & \texttt{0.4606553370833683839318043885447872940932} \\
\end{longtable}
\endgroup

\paragraph{Coefficients for our method of order 8.}

Denote the Butcher tableau as $(\matrix{A}, \vector{b}^\top, \vector{c})$.
$\matrix{A} = (a_{ij})$, $\vector{b} = (b_i)$, $\vector{c} = (c_i)$.

The precise parameters containing extended floating-point representations are defined as follows:

\begingroup
\small
\begin{longtable}{rl}
    \toprule
    Parameter & Precise Numerical Value \\
    \midrule
    \endhead
    \bottomrule
    \endfoot
        $c_2$ = $a_{21}$ = $c_3$ = $c_7$ = $c_{10}$ & \texttt{0.1726731646460114281008537718765708222154} \\
        $a_{31}$ = $a_{32}$ = $a_{13,11}$ & \texttt{0.0863365823230057140504268859382854111077} \\
        $a_{41}$ & \texttt{-0.2239109809347400004117529496080005305615} \\
        $a_{42}$ & \texttt{0.7239109809347400004117529496080005305615} \\
        $c_5$ = $c_8$ = $c_{11}$ = $c_{13}$ & \texttt{0.8273268353539885718991462281234291777846} \\
        $a_{51}$ = $a_{11,1}$ = $a_{13,1}$ & \texttt{0.3472041832132342858711439808030478211663} \\
        $a_{53}$ = $a_{11,3}$ = $a_{13,3}$ & \texttt{-0.3121452601573282541773983236104129679043} \\
        $a_{54}$ = $a_{11,4}$ = $a_{13,4}$ & \texttt{0.7922679122980825402054005709307943245226} \\
        $a_{61}$ = $a_{12,1}$ & \texttt{0.0086963396884199998627490167973331564795} \\
        $a_{63}$ = $a_{12,3}$ & \texttt{0.3685974296159088892091411830284448571034} \\
        $a_{64}$ = $a_{12,4}$ = $a_{12,9}$ & \texttt{0.1227062306956711109281098001742219864171} \\
        $a_{71}$ & \texttt{0.0763979083933828570644280095984760894169} \\
        $a_{73}$ = $a_{11,8}$ & \texttt{0.1015181575196660317721747904513016209880} \\
        $a_{74}$ & \texttt{-0.0052429012670374607357490281732068881895} \\
        $a_{81}$ & \texttt{-0.7558982639254171433277176566948577492132} \\
        $a_{83}$ & \texttt{0.9364357804719847625321949708312389037128} \\
        $a_{84}$ & \texttt{-2.3768037368942476206162017127923829735677} \\
        $a_{86}$ & \texttt{1.6982683863407949218897867670593032042510} \\
        $a_{87}$ & \texttt{0.8808802249164292069766394152756833481572} \\
        $a_{91}$ & \texttt{0.0763573206231600002745019664053336870410} \\
        $a_{93}$ & \texttt{-0.7371948592318177784182823660568897142068} \\
        $a_{94}$ & \texttt{-0.2454124613913422218562196003484439728342} \\
        $a_{95}$ & \texttt{-0.0272608833024616673872343284806675951493} \\
        $a_{97}$ & \texttt{0.9751775499691283340539009951473342618160} \\
        $a_{10,1}$ & \texttt{0.0664592344637600000784291332586667677260} \\
        $a_{10,3}$ = $a_{11,5}$ & \texttt{-0.1015181575196660317721747904513016209880} \\
        $a_{10,4}$ & \texttt{0.0052429012670374607357490281732068881895} \\
        $a_{10,5}$ & \texttt{0.0040043319862298409561247368066027703658} \\
        $a_{10,6}$ & \texttt{-0.0237373677735720641194965581328262053001} \\
        $a_{11,6}$ & \texttt{-0.6419111699433515349492526724849107715542} \\
        $a_{11,7}$ & \texttt{0.6242905203146565083547966472208259358085} \\
        $a_{11,9}$ & \texttt{0.6419111699433515349492526724849107715542} \\
        $a_{11,10}$ & \texttt{-0.6242905203146565083547966472208259358085} \\
        $a_{12,5}$ & \texttt{0.0030105722119316669068558872713336428276} \\
        $a_{12,6}$ & \texttt{-0.1227062306956711109281098001742219864171} \\
        $a_{12,7}$ & \texttt{-0.9098701805298291672671397181783341070689} \\
        $a_{12,8}$ & \texttt{-0.0030105722119316669068558872713336428276} \\
        $a_{12,10}$ & \texttt{0.9098701805298291672671397181783341070689} \\
        $a_{13,5}$ & \texttt{-0.0863365823230057140504268859382854111077} \\
        $a_{13,6}$ & \texttt{-1.2606885110156016333701275857118920804337} \\
        $a_{13,12}$ & \texttt{1.2606885110156016333701275857118920804337} \\
        $a_{14,5}$ & \texttt{-0.0676381975765999989324923528681467726183} \\
        $a_{14,6}$ & \texttt{0.7323606311023881494494908036993597015278} \\
        $a_{14,7}$ & \texttt{-0.5562780543026933337603363921860746242860} \\
        $a_{14,8}$ & \texttt{0.1039860419798822221154714575090368994840} \\
        $a_{14,9}$ & \texttt{0.0206135459281066663413310027788637783218} \\
        $a_{14,10}$ & \texttt{0.9451669431915822226492252810749635131749} \\
        $a_{14,11}$ & \texttt{0.1175136814952022219019699280826662540077} \\
        $a_{14,12}$ & \texttt{-0.5307519548082725935685995842560012576273} \\
        $a_{14,13}$ & \texttt{0.2350273629904044438039398561653325080154} \\
\end{longtable}
\endgroup

\paragraph{Coefficients for our optimized method of order 8.}
Denote the Butcher tableau as $(\matrix{A}, \vector{b}^\top, \vector{c})$.
$\matrix{A} = (a_{ij})$, $\vector{b} = (b_i)$, $\vector{c} = (c_i)$.

The precise parameters containing extended floating-point representations are defined as follows:

\begingroup
\small
\begin{longtable}{rl}
    \toprule
    Parameter & Precise Numerical Value \\
    \midrule
    \endhead
    \bottomrule
    \endfoot
        $b_1$ = $b_{14}$ & \texttt{4.99999999999999999999999999999998459e-02} \\
        $c_2$ = $a_{21}$ = $c_3$ = $c_7$ = $c_{10}$ & \texttt{1.72673164646011428100853771876570835e-01} \\
        $a_{31}$ & \texttt{8.63365823230057140504268859382852733e-02} \\
        $a_{32}$ & \texttt{8.63365823230057140504268859382855621e-02} \\
        $c_4$ = $c_6$ = $c_9$ = $c_{12}$ & \texttt{0.5} \\
        $a_{41}$ & \texttt{-2.23910980934740000411752949607999604e-01} \\
        $a_{42}$ & \texttt{-1.94610058394290906661483271413916266e-01} \\
        $a_{43}$ & \texttt{9.18521039329030907073236221021915678e-01} \\
        $c_5$ = $c_8$ = $c_{11}$ = $c_{13}$ & \texttt{8.27326835353988571899146228123427239e-01} \\
        $a_{51}$ = $a_{11,1}$ & \texttt{3.47204183213234285871143980803053631e-01} \\
        $a_{53}$ = $a_{11,3}$ = $a_{13,3}$ & \texttt{-3.12145260157328254177398323610416044e-01} \\
        $a_{54}$ = $a_{11,4}$ = $a_{13,4}$ & \texttt{7.92267912298082540205400570930791192e-01} \\
        $b_5$ = $b_8$ = $b_{11}$ = $b_{13}$ & \texttt{6.80555555555555555555555555555555384e-02} \\
        $a_{61}$ = $a_{12,1}$ & \texttt{9.31660703052918236121968330991439771e-03} \\
        $a_{63}$ = $a_{12,3}$ & \texttt{3.67400047375727443168913762088028958e-01} \\
        $a_{64}$ = $a_{12,4}$ & \texttt{1.235332538184833542594040221909992e-01} \\
        $a_{65}$ & \texttt{-2.4990822473997978953746758894255559e-04} \\
        $b_6$ = $b_9$ = $b_{12}$ & \texttt{1.18518518518518518518518518518518924e-01} \\
        $a_{71}$ & \texttt{-1.33173675734999488712059654478026333e-02} \\
        $a_{73}$ & \texttt{2.74707153333044796263479247308849614e-01} \\
        $a_{74}$ & \texttt{1.08626203179506537330664071859581443e-01} \\
        $a_{75}$ & \texttt{3.61466481093477826918414849170989833e-02} \\
        $a_{76}$ & \texttt{-2.33489472402387739313925066761157727e-01} \\
        $b_7$ = $b_{10}$ & \texttt{1.36111111111111111111111111111111077e-01} \\
        $a_{81}$ & \texttt{-7.55898263925417143327717656694715823e-01} \\
        $a_{83}$ & \texttt{9.36435780471984762532194970830611803e-01} \\
        $a_{84}$ & \texttt{-2.37680373689424762061620171279259544} \\
        $a_{85}$ & \texttt{4.44444444444444444444444444444443075e-01} \\
        $a_{86}$ & \texttt{1.69826838634079492188978676705958720} \\
        $a_{87}$ & \texttt{8.8088022491642920697663941527611491e-01} \\
        $a_{91}$ & \texttt{7.51167859389416352775606333800818414e-02} \\
        $a_{93}$ & \texttt{-7.34800094751454886337827524175789826e-01} \\
        $a_{94}$ & \texttt{-2.47066507636966708518808044382103171e-01} \\
        $a_{95}$ & \texttt{-5.53372058757747637507729438926114042e-02} \\
        $a_{96}$ & \texttt{4.58333333333333333333333333333329225e-01} \\
        $a_{97}$ & \texttt{9.75177549969128334053900995147236087e-01} \\
        $a_{98}$ & \texttt{2.85761390227930559426135505898599439e-02} \\
        $a_{10,1}$ & \texttt{1.56174510430642806014063108304982221e-01} \\
        $a_{10,3}$ & \texttt{-2.74707153333044796263479247308898725e-01} \\
        $a_{10,4}$ & \texttt{-1.08626203179506537330664071859481295e-01} \\
        $a_{10,5}$ & \texttt{-2.86037726718601566106090287612639699e-02} \\
        $a_{10,6}$ & \texttt{2.02167356316777136522005280486935188e-01} \\
        $a_{10,7}$ & \texttt{2.22222222222222222222222222222199967e-01} \\
        $a_{10,8}$ & \texttt{-3.53854345125778512510771934924967468e-03} \\
        $a_{10,9}$ & \texttt{7.58474831203853867242322814134885734e-03} \\
        $a_{11,5}$ & \texttt{-6.76097970480980249082804772955769433e-02} \\
        $a_{11,6}$ & \texttt{-7.14592594770649369099442123190306199e-01} \\
        $a_{11,7}$ & \texttt{6.24290520314656508354796647220625627e-01} \\
        $a_{11,8}$ & \texttt{6.76097970480980249082804772955769433e-02} \\
        $a_{11,9}$ & \texttt{7.14592594770649369099442123190306199e-01} \\
        $a_{11,10}$ & \texttt{-6.24290520314656508354796647220625627e-01} \\
        $a_{12,5}$ & \texttt{2.48471152903513119873851615240594623e-02} \\
        $a_{12,6}$ & \texttt{-9.38801470944387180938011431937732155e-02} \\
        $a_{12,7}$ & \texttt{-9.09870180529829167267139718178278317e-01} \\
        $a_{12,8}$ & \texttt{-1.81383613151035491089021755340167436e-02} \\
        $a_{12,9}$ & \texttt{9.38801470944387180938011431937732155e-02} \\
        $a_{12,10}$ & \texttt{9.09870180529829167267139718178278317e-01} \\
        $a_{12,11}$ & \texttt{-6.95866219998774266802045357898527425e-03} \\
        $a_{13,1}$ & \texttt{3.47204183213234285871143980803059794e-01} \\
        $a_{13,5}$ & \texttt{-1.21427404273258106125581474520309672e-01} \\
        $a_{13,6}$ & \texttt{-1.26068851101560163337012758571184525} \\
        $a_{13,11}$ & \texttt{1.21427404273258106125581474520309672e-01} \\
        $a_{13,12}$ & \texttt{1.26068851101560163337012758571184525} \\
        $c_{14}$ & \texttt{1.0} \\
        $a_{14,1}$ & \texttt{1.15059735858928208403907271827790529e-32} \\
        $a_{14,3}$ & \texttt{3.4074855701829383452570595988155705e-31} \\
        $a_{14,4}$ & \texttt{2.67509725499625842147317014859493508e-31} \\
        $a_{14,5}$ & \texttt{-6.02784742532391271376259265855368533e-02} \\
        $a_{14,6}$ & \texttt{7.83607001949023514488261749442154439e-01} \\
        $a_{14,7}$ & \texttt{-5.56278054302693333760336392186287998e-01} \\
        $a_{14,8}$ & \texttt{1.27894219614764308765368812757314217e-01} \\
        $a_{14,9}$ & \texttt{-3.06328249185286986974399429642385722e-02} \\
        $a_{14,10}$ & \texttt{9.4516694319158222264922528107480591e-01} \\
        $a_{14,11}$ & \texttt{8.62457805369592634572061465517691453e-02} \\
        $a_{14,12}$ & \texttt{-5.30751954808272593568599584255938345e-01} \\
        $a_{14,13}$ & \texttt{2.3502736299040444380393985616533733e-01} \\
\end{longtable}
\endgroup

\paragraph{Coefficients of our methods of order 10, 12 and 14}

The coefficient table of our methods of order 10, 12 and 14 are quite large, and we have made them public on GitHub, see \cite{he2026qdwithclustersminimalv1rktableaux}. A small piece of Julia script to load, use and test the order of the methods is also provided in the same repository. Try it if you are interested.

\appendix

\printbibliography

\end{document}